\documentclass[a4paper,reqno]{amsart}

\usepackage[foot]{amsaddr}
\usepackage{graphicx,enumerate,nicefrac,color,bm,cite}
\usepackage{amsmath, amssymb, amstext, amsfonts}
\usepackage{subfig}
\usepackage{pgfplots}

\usepackage{algorithm,algpseudocode,tabularx}

\makeatletter
\newcommand{\multiline}[1]{%
  \begin{tabularx}{\dimexpr\linewidth-\ALG@thistlm}[t]{@{}X@{}}
    #1
  \end{tabularx}
}
\makeatother

\newcommand{\norm}[1]{\left\|#1\right\|}
\newcommand{\operator}[1]{\mathsf{#1}}
\newcommand{\A}{\operator{A}}
\newcommand{\B}{\operator{B}}
\newcommand{\F}{\operator{F}}   

\renewcommand{\H}{\operator{H}}
\newcommand{\G}{\operator{G}}

\newcommand{\dt}{\,\mathsf{d} t}
\newcommand{\ds}{\,\mathsf{d} s}
\newcommand{\x}{\bm{\mathsf{x}}}
\newcommand{\nv}{\bm{\mathsf{n}}}
\newcommand{\dx}{\,\mathsf{d} \x}
\newcommand{\dprod}[1]{\left<#1\right>}
\newcommand{\muf}[1]{\mu\big(\left|\nabla #1 \right|^2\big)}
\DeclareMathOperator{\Div}{div}
\newcommand{\Ltwo}{\mathrm{L}^2}
\newcommand{\Hone}{\mathrm{H}^1}

\newtheorem{theorem}{Theorem}[section]

\newtheorem{proposition}[theorem]{Proposition} 
\newtheorem{corollary}[theorem]{Corollary} 

\theoremstyle{definition}
\newtheorem{remark}[theorem]{Remark}

\title[Modified Ka\v{c}anov method]{A modified Ka\v{c}anov iteration scheme with application to quasilinear diffusion models}

\author{Pascal Heid}
\email{pascal.heid@maths.ox.ac.uk}
\address{Mathematical Institute, University of Oxford, Oxford, OX2 6GG, UK}

\author{Thomas P.~Wihler}
\email{wihler@math.unibe.ch}
\address{Mathematics Institute, University of Bern, CH-3012 Bern, Switzerland}

\thanks{The authors acknowledge the financial support of the Swiss National Science Foundation (SNF), Grant No. 200021\underline{\space}182524, and Project No.~P2BEP2\underline{\space}191760.}



\begin{document}

\begin{abstract}
The classical Ka\v{c}anov scheme for the solution of nonlinear variational problems can be interpreted as a fixed point iteration method that updates a given approximation by solving a linear problem in each step. Based on this observation, we introduce a \emph{modified Ka\v{c}anov method}, which allows for (adaptive) damping, and, thereby, to derive a new convergence analysis under more general assumptions and for a wider range of applications. For instance, in the specific context of quasilinear diffusion models, our new approach does no longer require a standard monotonicity condition on the nonlinear diffusion coefficient to hold. Moreover, we propose two different adaptive strategies for the practical selection of the damping parameters involved.
\end{abstract}

\maketitle

\section{Introduction} \label{sec:introduction}

In this article we focus on a novel iterative Ka\v{c}anov type procedure for the solution of quasilinear elliptic partial differential equations (PDE) of the form 
\begin{subequations} \label{eq:scl}
\begin{alignat}{2}
-\Div \big\{\muf{u} \nabla u \big\}&=f  &\qquad&\text{in } \Omega \label{eq:sclmain} \\ 
u&=g && \text{on } \Gamma_1 \label{eq:scldirichlet} \\ 
-\muf{u} \nabla u \cdot \nv &= h &&\text{on } \Gamma_2, \label{eq:sclneumann}
\end{alignat}
\end{subequations}
where $\Omega \subset \mathbb{R}^d$, $d \in \{2,3\}$, is a non-empty, open, and bounded domain with a Lipschitz boundary $\partial \Omega$. We suppose that $\partial\Omega=\overline{\Gamma}_1 \cup \overline{\Gamma}_2$ is composed of a Dirichlet boundary part $\Gamma_1\neq\emptyset$ (of non-zero surface measure) and a Neumann boundary part $\Gamma_2$, and  $\nv$ denotes the unit outward normal vector on $\Gamma_2$. Moreover, $\mu$ is a real-valued diffusion coefficient, and $g$ and $h$ are Dirichlet and Neumann boundary condition functions, respectively. Nonlinear equations of this type are widely used in mathematical models of physical applications including, for instance, hydro- and gas-dynamics, as well as elasticity and plasticity, see, e.g., \cite{HaJeSh:97} and the references therein; we further refer to~\cite[\S69.2--69.3] {Zeidler:88} and~\cite[\S1.1]{Astala:09} for a discussion of the physical meaning.

For a given initial guess $u^0$ with $u^0=g$ on $\Gamma_1$, the traditional Ka\v{c}anov scheme for the solution of \eqref{eq:scl} is given by
\begin{subequations} \label{eq:kacanovs}
\begin{alignat}{2} 
-\Div \big\{\muf{u^n} \nabla u^{n+1} \big\}&=f  &\qquad &\text{in } \Omega \label{eq:kacanovsmain} \\ 
u^{n+1}&=g && \text{on } \Gamma_1 \label{eq:kacanovsdirichlet} \\ 
-\muf{u^n} \nabla u^{n+1} \cdot \nv &= h &&\text{on } \Gamma_2; \label{eq:kacanovsneumann}
\end{alignat}
\end{subequations}
this iteration scheme was originally introduced by Ka\v{c}anov in~\cite{kacanov:1959}, in the context of variational methods for plasticity problems. Observing that the above boundary value problem is a linear PDE for $u^{n+1}$ (given $u^n$), we can see Ka\v{c}anov's scheme as an \emph{iterative linearization method}, cf.~the general abstract iterative linearization methodology in~\cite{HeidWihler:2020a}.
In the literature, standard assumptions on the nonlinearity $\mu$, which guarantee the convergence of~\eqref{eq:kacanovs}, are expressed as follows, see, e.g. \cite{Zeidler:90,HaJeSh:97,GarauMorinZuppa:11,HeidWihler:2020a}:
\begin{enumerate}[($\mu$1)]
\item The diffusion function $\mu:\,[0,\infty) \to [0,\infty)$ is continuously differentiable;
\item The diffusion function $\mu$ is decreasing, i.e., $\mu'(t) \leq 0$ for all $t \geq 0$;
\item There are positive constants $m_\mu$ and $M_\mu$ such that
$m_\mu \leq \mu(t) \leq M_\mu$ for all $t \geq 0$;
\item There exists a positive constant $c_\mu >0$ such that
$2 \mu'(t^2)t^2+\mu(t^2) \geq c_\mu$ for all  $t \geq 0$; if we let 
\begin{align} \label{eq:orlicz}
\phi(t):=\int_0^t\mu(s^2)s \ds, \qquad t \geq 0,
\end{align} 
then we note that this condition means that $\phi$ is strictly convex.

\end{enumerate}
Examples satisfying the assumptions ($\mu$1)--($\mu$4) can be found, e.g., in~\cite{HaJeSh:97} and the references therein. \\

In this work we address the open question of whether the monotonicity assumption ($\mu$2) is necessary or not for the convergence of the Ka\v{c}anov scheme. Indeed, numerical experiments in~\cite{GarauMorinZuppa:11,HeidWihler:2020a} indicate that it may be dropped. Based on this observation, in this work, we will introduce a \emph{modified Ka\v{c}anov iteration method that converges without imposing condition} ($\mu$2), and, thereby, allows for the application to a considerably wider range of physical models. For instance, in the context of quasi-Newtonian fluids, the analysis of the traditional Ka\v{c}anov method is limited to shear-thinning materials corresponding to a decreasing viscosity coefficient, whilst our new scheme can be applied, in addition, to shear-thickening substances, which have an increasing viscosity coefficient $\mu$. We note that the classical Ka\v{c}anov method was already applied to incompressible generalized Newtonian fluid flow problems with a power-law like rheology of possibly shear-thickening fluids in~\cite{Carelli:2010}; in that work, however, very restrictive conditions needed to be imposed in order to show the convergence of the sequence of iterates to a solution of a regularized problem. Moreover, advanced convergence results of the Ka\v{c}anov scheme for a relaxed $p$-Poisson problem, for $1<p \leq 2$, can be found in~\cite{Diening:2020}. Due to the assumption on $p$, we point out that the analysis in that work is restricted to decreasing diffusion coefficients $\mu$, and, moreover, the coefficient is bounded from below and above by the lower and upper cut-off parameters from the truncation considered in the relaxed $p$-Poisson problem. The convergence analysis of the Ka\v{c}anov scheme (but not the overall analysis) in~\cite{Diening:2020} has been developed further and generalized to a broader class of decreasing diffusion coefficients, which correspond to the viscosity function of shear-thinning fluids, in~\cite{HeidSuli:21}. We note that the analysis in~\cite{HeidSuli:21} can indeed also be applied to the generalized Stokes problem, cf.~\cite[\S 4.1]{HeidSuli2:21}, where the convergence is shown for the Bercovier--Engelman regularization of steady Bingham fluids. We emphasize once more that a key prerequisite in the convergence analysis of the Ka\v{c}anov scheme in~\cite{Diening:2020, HeidSuli:21,HeidSuli2:21}, as well as in the classical proof, is the monotonicity of the diffusion coefficient; in fact, in all the convergence proofs of the Ka\v{c}anov scheme we are aware of, except for the one in~\cite{Carelli:2010}, it is shown that an underlying (energy) functional decays along the sequence generated by the Ka\v{c}anov scheme, for which, again, it is standard to assume that the diffusion coefficient is monotonically decreasing. The proof in the present work is also based on the decay of the energy functional, which, however, can be obtained without imposing the assumption $(\mu 2)$ when a damping parameter is introduced. The key idea in devising the modified scheme together with the proof of convergence is based on the fact that~\eqref{eq:kacanovs} can be cast into the unified iteration scheme introduced in~\cite{HeidWihler:2020a}, see also \cite{HeidWihler:2020b}. To sketch the idea, for $\Gamma_2=\emptyset$, upon defining the PDE residual
\[
\F:=-\Div \big\{\muf{(\cdot)} \nabla (\cdot) \big\}-f, 
\] 
and, for given $u$, the \emph{linear preconditioning operator}
\[
\A(u)(\cdot):=-\Div \big\{\muf{u} \nabla (\cdot) \big\},
\]
the iterative procedure~\eqref{eq:kacanovs} can be written (formally) in terms of the \emph{fixed point iteration}
\[
u^{n+1}=u^n-\A(u^n)^{-1}\F(u^n),\qquad n\ge 0.
\]
With the aim of obtaining an improved control of the updates in each step, we introduce a \emph{step size parameter} $\delta(u^n)>0$ in the iteration, viz.
\[
u^{n+1}=u^n-\delta(u^n)\A(u^n)^{-1}\F(u^n),\qquad n\ge 0.
\]
This yields the \emph{modified Ka\v{c}anov method} proposed and analyzed in this work.

\subsubsection*{Outline} 

We begin by deriving an appropriate framework for abstract nonlinear variational problems in \S\ref{sec:abstract}. In particular, we introduce a modified version of the classical Ka\v{c}anov iteration scheme, and prove a new convergence result under assumptions that are milder than in the classical setting. The purpose of \S\ref{sec:stepsize} is to devise two different adaptive strategies for the selection of the damping parameters in the modified method. Subsequently, our general theory is applied to quasilinear diffusion models in \S\ref{sec:dampedkacanov}, which also contains a numerical study within the framework of finite element discretizations. Finally, we add some concluding remarks in~\S\ref{sec:conclusion}.

\section{Abstract analysis}\label{sec:abstract}

Throughout, $Y$ is a reflexive real Banach space, equipped with a norm denoted by $\|\cdot\|_Y$, and $K\subset Y$ is a closed, convex subset. 

\subsection{Nonlinear variational problems}

Consider a (nonlinear) G\^ateaux continuously differentiable functional $\H:\,K\to\mathbb{R}$ that has a strongly monotone G\^ateaux derivative, i.e., there exists a constant $\nu >0$ such that 
\begin{align} \label{eq:strongmonotonicity}
\dprod{\H'(u)-\H'(v),u-v} \geq \nu \norm{u-v}_Y^2 \qquad \forall u,v \in K,
\end{align}
where $\dprod{\cdot,\cdot}$ is the duality product on $Y^\star\times Y$, with $Y^\star$ signifying the dual space of~$Y$.

\begin{proposition}\label{pr:varabstract}
Suppose that $\H:\,K\to\mathbb{R}$ is a (G\^ateaux) continuously differentiable functional that satisfies the strong monotonicity condition~\eqref{eq:strongmonotonicity}. Then, there exists a unique minimizer $u^\star\in K$ of $\H$ in $K$, i.e. $\H(u^\star)\le \H(v)$ for all $v\in K$. Furthermore, $u^\star\in K$ is the unique solution of the weak inequality  
\begin{align} \label{eq:variationalinequality}
\dprod{\H'(u^\star),v-u^\star} \geq 0 \qquad \forall v \in K.
\end{align}
\end{proposition}

\begin{proof}
We follow along the lines of the proof of~\cite[Thm.~25.L]{Zeidler:90}. 
\begin{enumerate}[1.]
\item Since $\H$ is a G\^ateaux continuously differentiable functional on $K$, it is, in particular, continuous on $K$. Moreover, from~\eqref{eq:strongmonotonicity} we infer that $\H$ is strictly convex, see, e.g., \cite[Prop.~25.10]{Zeidler:90}. These two properties, in turn, imply that $\H$ is weakly sequentially lower semicontinuous, see~\cite[Prop.~25.20]{Zeidler:90}. 
\item If the set $K$ is bounded, then the functional $\H$, being weakly sequentially lower semicontinuous, has a minimum on $K$, see~\cite[Thm.~25.C]{Zeidler:90}. Otherwise, if $K$ is unbounded, then we show that $\H$ is weakly coercive. To this end, take any $u,v \in K$, and define the scalar function 
\begin{equation}\label{eq:phi}
\varphi(t):=\H(u+t(v-u)),\qquad t \in [0,1];
\end{equation} 
since $K$ is convex, note that $u+t(v-u) \in K$ for all $t \in [0,1]$.  Applying the chain rule reveals that
\begin{align}\label{eq:7}
\varphi'(t)=\dprod{\H'(u+t(v-u)),v-u}.
\end{align}
Thus, by virtue of the fundamental theorem of calculus, we deduce that
\begin{align*}
\H(v)-\H(u)&=\int_0^1 \dprod{\H'(u+t(v-u)),v-u} \dt \\
&=\int_0^1 \dprod{\H'(u+t(v-u))-\H'(u),v-u} \dt + \dprod{\H'(u),v-u}.
\end{align*}
Therefore, exploiting \eqref{eq:strongmonotonicity}, it follows that
\begin{align*}
\H(v)-\H(u) \geq \frac{\nu}{2} \norm{v-u}_Y^2-\norm{\H'(u)}_{Y^\star}\norm{v-u}_Y.
\end{align*} 
Hence, we see that $\H(v) \to \infty$ for $\norm{v}_Y \to \infty$, i.e., $\H$ is weakly coercive on $K$. Then, owing to~\cite[Thm.~25.D]{Zeidler:90}, we conclude that $\H$ has a minimum $u^\star$ in $K$. We note that this minimum is unique since $\H$ is strictly convex. 
\item Finally, if $u^\star$ is the minimum of $\H$ in $K$, then the function $\varphi(t)$ from~\eqref{eq:phi}, with $u=u^\star$, has a minimum at $t=0$. This implies that $\varphi'(0) \geq 0$. In turn, exploiting~\eqref{eq:7}, this holds true if and only if $\dprod{\H'(u^\star),v-u^\star} \geq 0$, which yields~\eqref{eq:variationalinequality}. 
Conversely, since $\H'$ is strongly monotone, cf.~\eqref{eq:strongmonotonicity}, and $u^\star$ satisfies the weak inequality~\eqref{eq:variationalinequality}, the function $\varphi(t)$ is increasing on $[0,1]$. In fact, for $t\in(0,1]$, we have
\begin{align*}
\varphi'(t)
&=\frac{1}{t}\dprod{\H'(u^\star+t(v-u^\star))-\H'(u^\star),t(v-u^\star)}
+\dprod{\H'(u^\star),v-u^\star}\\
&\ge\frac{1}{t}\dprod{\H'(u^\star+t(v-u^\star))-\H'(u^\star),t(v-u^\star)}
\ge\nu t\|v-u^\star\|^2_Y\ge0.
\end{align*}
Therefore, we obtain
$
\H(v)=\varphi(1)\ge\varphi(0)
=\H(u^\star),
$
i.e., $u^\star$ is the minimum of $\H$ in $K$ since $v$ was arbitrary. 
\end{enumerate}
This completes the proof.
\end{proof}

\begin{remark}
We emphasize that the application of~\cite[Thms.~25.C \& 25.D]{Zeidler:90} in the proof of Proposition~\ref{pr:varabstract} require the space $Y$ to be reflexive. More precisely, these results make use of the fact that every bounded sequence in a reflexive Banach space has a weakly convergent subsequence. 
\end{remark}

Consider the following assumption on the closed and convex subset $K$: 
\begin{enumerate}[(K)]
\item The set $X:=\{u-v:\, u,v\in K\}$ is a linear closed subspace of $Y$, and $x+y \in K$ for all $x\in X$ and $y\in K$.
\end{enumerate}

\begin{corollary}\label{cor:1}
Suppose that the subset $K$ has property~{\rm(K)}, and let the assumptions of Proposition~\ref{pr:varabstract} hold true. Then, the unique minimizer $u^\star\in K$ of $\H$ satisfies 
\begin{equation}\label{eq:variationalproblem}
\dprod{\H'(u^\star),v}=0\qquad\forall v\in X.
\end{equation}
\end{corollary}


\begin{proof}
Let $u^\star\in K$ be the unique minimizer of $\H$ on $K$. Then, for
any $v\in X$, owing to property (K), it holds that $v+u^\star \in K$. Thus, using ~\eqref{eq:variationalinequality}, we have
$0\le\dprod{\H'(u^\star),(v+u^\star)-u^\star}=\dprod{\H'(u^\star),v}$.
Similarly, upon replacing $v$ by $-v$, we infer that 
$0\le-\dprod{\H'(u^\star),v}$,
which concludes the argument.
\end{proof}

\subsection{Modified Ka\v{c}anov method}
 
We consider mappings $a:\,K\times Y\times X\to\mathbb{R}$ and $b:\,K\times X\to\mathbb{R}$, with $X$ being the linear subspace from property (K) above, which satisfy the following properties:
\begin{enumerate}[({A}1)]
\item For any given $u\in K$, we suppose that $a(u;\cdot,\cdot)$ is a bilinear form on $Y \times X$, and $b(u,\cdot)\in X^\star$; in the sequel, we use the notation $b(u,\cdot)=\dprod{b(u),\cdot}$, where the dual product is evaluated on the space $X^\star\times X$. 
\item There exist positive constants $\alpha,\beta>0$ such that, for any $u\in K$, the form $a(u;\cdot,\cdot)$ is uniformly bounded on $Y\times X$ and coercive on $X\times X$ in the sense that 
\begin{align} \label{eq:abounded}
a(u;v,w) \leq \beta \norm{v}_Y \norm{w}_Y \qquad \forall v\in Y,~\forall w\in X, 
\end{align}
and 
\begin{align} \label{eq:acoercive}
a(u;v,v)\ge\alpha\norm{v}^2_Y\qquad\forall v\in X,
\end{align}
respectively; in particular, if the set $K$ satisfies property (K), then it follows that
\begin{align} \label{eq:coerciveC}
a(u;v-w,v-w) \geq \alpha \norm{v-w}_Y^2\qquad\forall v,w \in K.
\end{align}
\item There are G\^ateaux continuously differentiable functionals $\G:\,K \to \mathbb{R}$ and $\B:\,K \to \mathbb{R}$ such that, for all $u\in K$, it holds $\G'(u)|_X=a(u;u,\cdot)$ and $\B'(u)|_X=b(u)$ in $X^\star$.
\item The (continuously differentiable) functional $\H:\,K\to\mathbb{R}$ given by $\H(u):=\G(u)-\B(u)$, $u \in K$, with $\G$ and $\B$ from (A3), satisfies the strong monotonicity condition~\eqref{eq:strongmonotonicity}.
\end{enumerate}

If the closed and convex subset $K\subset Y$ fulfils property (K), then the unique minimizer $u^\star\in K$ of the functional $\H$ from (A4) solves the weak formulation
\begin{equation}\label{eq:wf}
0=\dprod{\H'(u^\star),v}
=\dprod{\G'(u^\star)-\B'(u^\star),v}
=a(u^\star;u^\star,v)-\dprod{b(u^\star),v} \qquad \forall v \in X,
\end{equation}
cf. Corollary~\ref{cor:1}. Now, for given $u\in K$, define the \emph{linear} operator $\A(u):\,Y\to X^\star$, $v\mapsto \A(u)v$, by
\[
\dprod{\A(u)v,w}=a(u;v,w) \qquad \forall w \in X.
\]
Then, the weak formulation~\eqref{eq:wf} can be expressed by
\[
\A(u^\star)u^\star=b(u^\star)\qquad\text{in }X^\star.
\]
In light of (A2), for any $u\in K$, we notice that $a(u;\cdot,\cdot)$ is a bounded and coercive bilinear form on the closed subspace $X \times X$. In particular, thanks to the Lax-Milgram theorem, for any $u\in K$ and $\ell\in X^\star$, we conclude that there exists a unique $w_{u,\ell} \in X$ such that $\A(u)w_{u,\ell} =\ell$ in $X^\star$, i.e., $\A(u)|_X:X \to X^\star$ is invertible for any $u \in K$. Hence, noticing that 
\begin{equation}\label{eq:FH'}
\F(u):=\H'(u)=\A(u)u-b(u)\in X^\star,
\end{equation} 
the classical Ka\v{c}anov method in abstract form, for given $u^n\in K$, reads as
\begin{subequations}\label{eq:Ktrad}
\begin{align}
u^{n+1}&=u^n-\rho^n,\qquad n\ge 0,\\
\intertext{where $\rho^n\in X$ is uniquely defined through}
\A(u^n)\rho^{n}&=\F(u^n)\quad\text{in }X^\star.\label{eq:Ktrad2}
\end{align}
\end{subequations}
A modification of this procedure is obtained by invoking a parameter $\delta:\, K \to (0,\infty)$, thereby yielding the new scheme
\begin{align}\label{eq:fixedpointkacanovdamped}
u^{n+1}&=u^{n}-\delta(u^n)\rho^n,\qquad n\ge 0,
\end{align}
with $\rho^n$ as in~\eqref{eq:Ktrad2}.
Equivalently, upon introducing the forms
\begin{equation}\label{eq:forms}
a_{\delta(u)}(u;v,w):=\frac{1}{\delta(u)}a(u;v,w),\qquad
b_{\delta(u)}(u):=\frac{1}{\delta(u)}a(u;u,\cdot)-\F(u),
\end{equation}
for $u\in K, v\in Y,$ and $w\in X$, we derive the \emph{modified Ka\v{c}anov iteration} in weak form:
\begin{equation}\label{eq:Knew}
u^{n+1} \in K: \qquad a_{\delta^n}(u^n;u^{n+1},v)=\dprod{b_{\delta^n}(u^n),v}\qquad\forall v\in X,\qquad n\ge 0,
\end{equation}
where we use the notation $\delta^n:=\delta(u^n)$. Clearly, for $\delta \equiv 1$, the traditional Ka\v{c}anov scheme~\eqref{eq:Ktrad} is recovered.

\begin{proposition} \label{prop:welldefinedkacanov}
Suppose {\rm(A1)--(A4)}, and that $K$ satisfies property {\rm(K)}. Then, for any initial guess $u^0 \in K$, the modified Ka\v{c}anov iteration~\eqref{eq:Knew} remains well-defined for each $n\ge 0$, i.e., for given $u^n\in K$, the solution $u^{n+1}\in K$ of the weak formulation~\eqref{eq:Knew} exists and is unique.
\end{proposition}

\begin{proof}
For fixed $u^n\in K$, the solution $\rho^n\in X$ of~\eqref{eq:Ktrad2} exists and is unique since $\A(u^n)|_X:X \to X^\star$ is invertible. Moreover, owing to property (K), we infer that $u^{n+1}\in K$ in~\eqref{eq:fixedpointkacanovdamped}.
\end{proof}

\subsection{Convergence analysis}\label{sec:convergence}

We are now in the position to state and prove the main result of our work.

\begin{theorem} \label{thm:convergenceNew}
Given {\rm (A1)--(A4)} and {\rm (K)}. We further assume the following conditions:
\begin{enumerate}[\rm(a)]
\item $\H'$ is continuous with respect to the weak topology on $X^\star$ in the sense that, for any sequence $\{z^n\}_{n\ge 0}\subset K$ with a limit $z^\star\in K$, it holds that 
\begin{equation}\label{eq:Hweak}
\lim_{n\to\infty}\dprod{\H'(z^n),w}=\dprod{\H'(z^\star),w}\qquad \forall w\in X;
\end{equation}
\item there exists a damping strategy such that $\delta(u^n)\geq \delta_{\min}>0$ and
\begin{align} \label{eq:keyinequalityeps}
\H(u^{n})-\H(u^{n+1}) \geq  \gamma \norm{u^{n+1}-u^n}_Y^2 \qquad \forall n \geq 0,
\end{align}
for some constants $\delta_{\min},\gamma>0$ independent of $n$.
\end{enumerate}
Then, the damped Ka\v{c}anov iteration~\eqref{eq:fixedpointkacanovdamped} converges to the unique solution $u^\star \in K$ of~\eqref{eq:variationalproblem} for any initial guess $u^0 \in K$.
\end{theorem}

\begin{proof}
We will proceed in three steps: First, we show that the difference of two consecutive iterates, i.e. $\norm{u^{n+1}-u^n}_Y$, tends to zero as $n \to \infty$. Subsequently, we will verify the convergence of $\{u^n\}_{n \geq 0}$, and finally that the limit equals to $u^\star$. For this purpose, we will essentially follow along the lines of the proof of~\cite[Prop.~2.1]{HeidWihler:2020a}; see also the closely related argument in~\cite[Thm.~25.L]{Zeidler:90}.
\begin{enumerate}[1.]
\item In light of Proposition~\ref{pr:varabstract} we recall that $\H$ is bounded from below by $\H(u^\star)$. Moreover, $\{\H(u^n)\}_{n \geq 0}$ is decreasing thanks to the assumption~\eqref{eq:keyinequalityeps}. Hence, this sequence converges, and we conclude that
\begin{align*}
0 \leq \gamma \norm{u^{n+1}-u^n}_Y^2 \leq \H(u^{n+1})-\H(u^{n}) \to 0 \qquad \text{as} \ n \to \infty;
\end{align*}
i.e., $\lim_{n \to \infty} \norm{u^{n+1}-u^n}_Y=0$.
\item Next, we shall verify the existence of the limit of the sequence $\{u^n\}_{n \geq 0}$. By the strong monotonicity~\eqref{eq:strongmonotonicity}, for any $m \geq n \geq 0$, it holds that 
\begin{align*}
\nu \norm{u^m-u^n}_Y^2 \leq \dprod{\H'(u^m)-\H'(u^n),\epsilon_{m,n}}, 
\end{align*}
with $\epsilon_{m,n}:=u^m-u^n\in X$. Combining~\eqref{eq:FH'} and~\eqref{eq:forms}, for $\delta=\delta(u)>0$, we note that  
\begin{equation}\label{eq:uK}
\H'(u)
=\F(u)
=a_\delta(u;u,\cdot)-b_\delta(u)\qquad\forall u\in K 
\end{equation}
in $X^\star$, and thus,
\begin{align*}
\nu \norm{\epsilon_{m,n}}_Y^2 & \leq a_{\delta^m}(u^m;u^m,\epsilon_{m,n})-\dprod{b_{\delta^m}(u^m),\epsilon_{m,n}} \\ & \quad -a_{\delta^n}(u^n;u^n,\epsilon_{m,n})+\dprod{b_{\delta^n}(u^n),\epsilon_{m,n}}.
\end{align*}
Using~\eqref{eq:Knew}, this further leads to 
\begin{align*}
\nu \norm{\epsilon_{m,n}}_Y^2  &\leq a_{\delta^m}(u^m;u^m-u^{m+1},\epsilon_{m,n})-a_{\delta^n}(u^n;u^n-u^{n+1},\epsilon_{m,n}).
\end{align*}
Applying~\eqref{eq:abounded} yields
\begin{align*}
\nu \norm{\epsilon_{m,n}}_Y^2 & \leq  \frac{\beta}{\delta_{\min}} \norm{\epsilon_{m,n}}_Y \left(\norm{u^m-u^{m+1}}_Y+\norm{u^n-u^{n+1}}_Y\right),
\end{align*} 
and thus
\[
\norm{\epsilon_{m,n}}_Y \leq \frac{\beta}{\nu \delta_{\min}} \left(\norm{u^m-u^{m+1}}_Y+\norm{u^n-u^{n+1}}_Y\right). 
\]
From the first step of the proof, we conclude that $\{u^n\}_{n \geq 0}$ is a Cauchy sequence in $K$. Since $K$ is a closed subset of a Banach space, the sequence $\{u^n\}_{n \geq 0}$ has a limit $\mathfrak{u} \in K$. 
\item It remains to verify that $\mathfrak{u}$ is a solution of~\eqref{eq:variationalproblem}. To this end, from~\eqref{eq:uK} and~\eqref{eq:Knew} it follows that
\begin{align*}
a_{\delta^n}(u^n;u^{n+1}-u^n,w)+\dprod{\H'(u^n),w}
=a_{\delta^n}(u^n;u^{n+1},w)-\dprod{b_{\delta^n}(u^n),w}
=0,
\end{align*}
for all $w\in X$. Recalling that $\{u^{n+1}-u^n\}_{n \geq 0}$ is a vanishing sequence in $X$, and exploiting~\eqref{eq:abounded}, we have that $a_\delta(u^n;u^{n+1}-u^n,w) \to 0$ as $n \to \infty$, for any $w\in X$. Moreover, by the weak continuity property~\eqref{eq:Hweak} of $\H'$, we obtain
\begin{align*} 
\dprod{\H'(\mathfrak{u}),w}
=\lim_{n\to0}\dprod{\H'(u^n),w} 
= 0\qquad\forall w \in X,
\end{align*}
i.e., $\mathfrak{u}$ is a solution of~\eqref{eq:variationalproblem}. Since the solution is unique thanks to Proposition~\ref{pr:varabstract} and Corollary~\ref{cor:1}, we infer that $\mathfrak{u}=u^\star$. This completes the argument.
\end{enumerate}
\end{proof}

\begin{remark} \label{rem:classical}
The classical convergence theory for the (standard) Ka\v{c}anov method requires the following key inequality to hold:
\begin{align} \label{eq:keyinequalityZ}
\G(u)-\G(v) \geq \frac{1}{2} (a(u;u,u)-a(u;v,v))\qquad \forall u,v \in K;
\end{align} 
see, e.g.,~\cite[Thm.~25.L and Eq.~(106)]{Zeidler:90}. In order to verify~\eqref{eq:keyinequalityZ} in the context of our model problem~\eqref{eq:scl}, the monotonicity assumption $(\mu 2)$ is decisive. On the contrary, the analysis in our present work is based on the bound~\eqref{eq:keyinequalityeps} which, in turn, allows to omit the monotonicity of the (nonlinear) diffusion coefficient $\mu$ in the application to quasilinear elliptic PDE~\eqref{eq:scl}; see Theorem~\ref{thm:sclconvergence} below. Furthermore, in contrast to the traditional framework, the operator $\B$ from assumption (A3) does not need to be linear in our analysis, and, in addition, the symmetry of $a(u,\cdot,\cdot)$ is no longer necessary. We remark that these improvements come at the expense of condition (K) as well as of the crucial estimate~\eqref{eq:keyinequalityeps}; in the context of quasilinear elliptic PDE~\eqref{eq:scl}, however, these assumptions do not implicate any drawback. Finally, we note that if $\B$ is linear, then~\eqref{eq:keyinequalityZ} implies \eqref{eq:keyinequalityeps} with $\gamma=\nicefrac{\alpha}{2}$.
\end{remark}

The next result states that if $\H'$ is Lipschitz continuous in the sense that there exists a constant $L_{\H}>0$ such that
\begin{align} \label{eq:lipschitz}
\dprod{\H'(u)-\H'(v),u-v}\leq L_{\H} \norm{u-v}_Y^2 \qquad\forall u,v \in K,
\end{align}
then our key assumption~\eqref{eq:keyinequalityeps} from Theorem~\ref{thm:convergenceNew} is satisfied for sufficiently small damping parameters.

\begin{corollary} \label{cor:convergence}
Assume {\rm (A1)--(A4)} and {\rm (K)}, and suppose that~\eqref{eq:lipschitz} holds. Then, we have the estimate
\begin{equation*} 
\begin{split}
\H(u^{n})-\H(u^{n+1}) 
& \geq \left(\frac{\alpha}{\delta^n}-\frac{L_\H}{2}\right) \norm{u^{n+1}-u^n}_Y^2 >0, 
\end{split}
\end{equation*}
where $\alpha>0$ is the constant from~\eqref{eq:acoercive}. In particular, if $\delta^n \in [\delta_{\min},\delta_{\max}]$, with $0<\delta_{\min} \leq \delta_{\max} < \nicefrac{2 \alpha}{L_\H}$, for all $n \geq 0$, and $\H'$ is continuous with respect to the weak topology on $X^\star$, cf.~condition~{\rm(a)} from Theorem~\ref{thm:convergenceNew}, then the damped Ka\v{c}anov iteration~\eqref{eq:fixedpointkacanovdamped} converges to the unique solution $u^\star \in K$ of~\eqref{eq:variationalproblem} for any initial guess $u^0 \in K$.
\end{corollary}

\begin{proof}
Our argument follows along the lines of the proof of~\cite[Thm.~2.6]{HeidWihler:2020a}, however, with a different bilinear form on account of the present iteration scheme~\eqref{eq:Knew}. Similarly as in the proof of Proposition~\ref{pr:varabstract}, we define the scalar function $\varphi(t):=\H(u^n+t(u^{n+1}-u^n))$, for $t \in [0,1]$ and $n \geq 0$. Then, we find that
\begin{align}
\H(u^{n+1})-\H(u^n) 
&=\int_0^1 \dprod{\H'(u^n +t(u^{n+1}-u^n)),u^{n+1}-u^n} \dt\label{eq:Hint}\\
&= \int_0^1 \dprod{\H'(u^n +t(u^{n+1}-u^n))-\H'(u^n),u^{n+1}-u^n} \dt\nonumber\\
&\quad +\dprod{\H'(u^n),u^{n+1}-u^n}.\nonumber
\end{align}
Hence, by invoking the Lipschitz continuity~\eqref{eq:lipschitz}, the identity~\eqref{eq:uK}, and the modified Ka\v{c}anov scheme~\eqref{eq:Knew}, we obtain 
\begin{align*}
\H(u^{n+1})&-\H(u^n)\\ 
&\leq \frac{L_\H}{2} \norm{u^{n+1}-u^{n}}_Y^2  +a_{\delta^n}(u^n;u^n,u^{n+1}-u^n)
-\dprod{b_{\delta^n}(u^n),u^{n+1}-u^n}\\
&= \frac{L_\H}{2} \norm{u^{n+1}-u^{n}}_Y^2 -a_{\delta^n}(u^n;u^{n+1}-u^n,u^{n+1}-u^n).
\end{align*}
Furthermore, employing the coercivity assumption~\eqref{eq:coerciveC}, it follows that
\begin{align*}
\H(u^{n+1})-\H(u^n) 
& \leq \left(\frac{L_\H}{2} - \frac{\alpha}{\delta^n}\right) \norm{u^{n+1}-u^{n}}_Y^2.
\end{align*}
Moreover, if $\delta^n \leq \delta_{\max}$ for all $n \geq 0$, then we further deduce the bound
\begin{equation} \label{eq:keyinequality3}
\begin{split}
\H(u^{n})-\H(u^{n+1}) 
& \geq \left(\frac{\alpha}{\delta_{\max}}-\frac{L_\H}{2}\right) \norm{u^{n+1}-u^n}_Y^2 
\end{split}.
\end{equation} 
If $\delta_{\max}<\nicefrac{2\alpha}{L_\H}$, then $\left(\nicefrac{\alpha}{\delta_{\max}}-\nicefrac{L_\H}{2}\right)>0$, and, in turn, Theorem~\ref{thm:convergenceNew} implies the convergence of the sequence $\{u^n\}_{n \geq 0}$ to $u^\star$.
\end{proof}

\begin{remark} \label{rm:contraction}
Applying the abstract analysis in~\cite{HeidPraetoriusWihler:2020}, given the assumptions of Theorem~\ref{thm:convergenceNew}, it can be shown that the iterates generated by the modified Ka\v{c}anov scheme~\eqref{eq:Knew} satisfy the contraction property 
\[
\H(u^{n+1})-\H(u^\star) \leq q \left(\H(u^n)-\H(u^\star)\right) \qquad\forall n\ge 0,
\]
for some constant $0<q<1$, where $u^\star$ is the solution of~\eqref{eq:variationalproblem}. In particular, in view of Corollary~\ref{cor:convergence} with a constant damping parameter $\delta=\delta^n \in\left(0,\nicefrac{2 \alpha}{L_\H}\right)$ for all $n\ge 0$, we have that 
\begin{align*}
q(\delta)=\left(1-\frac{2\delta\nu^2 \left(\alpha-\nicefrac{\delta L_\H}{2}\right)}{\beta^2 L_{\H}}\right),
\end{align*}
cf.~\cite[Thm.~2.1]{HeidPraetoriusWihler:2020}. By taking the derivative with respect to $\delta$, it follows immediately that the minimum is attained at $\delta=\nicefrac{\alpha}{L_\H}$; noticing that the derivation of $q$ involves some rough estimates, however, this choice is typically suboptimal with regards to the convergence rate.  
\end{remark}

\section{Adaptive step size control} \label{sec:stepsize}

In this section, we will present two adaptive methods for selecting the damping parameter $\delta^n$ in the modified Ka\v{c}anov iteration~\eqref{eq:Knew}. To this end, recall the key inequality~\eqref{eq:keyinequalityeps} from Theorem~\ref{thm:convergenceNew}, and set $\delta_{\max}:=\nicefrac{\alpha}{L_{\H}}$ in~\eqref{eq:keyinequality3} (which is a possibly pessimistic choice as mentioned in Remark~\ref{rm:contraction}); then \eqref{eq:keyinequalityeps} holds for $\gamma=\nicefrac{L_\H}{2}$. Alternatively, from Remark~\ref{rem:classical}, we recall that within the setting of the classical Ka\v{c}anov scheme, i.e., for $\delta\equiv 1$, the bound~\eqref{eq:keyinequalityeps} can be shown for the constant $\gamma=\nicefrac{\alpha}{2}$ under more restrictive assumptions on the nonlinearity. This observation may suggest that a smaller choice of $\gamma$ potentially relates to a larger size of the damping parameter. We thus propose that the sequence $\{u^n\}_{n\ge 0}$ is required to satisfy an estimate of the form
\begin{align} \label{eq:lowerbound}
\H(u^{n})-\H(u^{n+1}) \geq \theta  \min \left\{ \alpha,L_\H \right\}\norm{u^{n+1}-u^n}_Y^2,
\end{align}
for a constant $0<\theta\le \nicefrac12$, which still guarantees the convergence of the modified Ka\v{c}anov scheme~\eqref{eq:Knew} in regard to Theorem~\ref{thm:convergenceNew} without imposing an upper bound on the damping parameter. In our numerical experiments in \S\ref{sec:numexp}, we let $\theta=0.1$. Moreover, in order to prevent too small steps, we set $\delta_{\min}:=\nicefrac{\alpha}{4 L_\H}$, which, in view of Remark~\ref{rm:contraction}, is a reasonable choice. We emphasize that the constants $\alpha$ and $L_\H$ must both be known a priori; in particular, we  assume that $\H'$ is Lipschitz continuous as proposed in~\eqref{eq:lipschitz}.

The two adaptive step size procedures to be presented below both pursue a similar strategy, namely, to maximize the difference $\H(u^n)-\H(u^{n+1})$ in each step by choosing an appropriate step size $\delta^n=\delta(u^n) \geq \delta_{\min}$. Indeed, recalling that $u^\star$ is the unique minimizer of $\H$ in $K$, it seems obvious that a maximal decay of the functional $\H$ along the sequence $\{u^n\}_{n \geq 0}$ will potentially accelerate the convergence of $\{u^n\}_{n \geq 0}$ to $u^\star$.

\subsection{Step size control via Taylor expansion} \label{sec:sstaylor}

We begin by recalling~\eqref{eq:Hint}, which in regard to~\eqref{eq:FH'}, can be stated as
\begin{align} \label{eq:hdfifftaylor}
\H(u^{n+1})-\H(u^n)=\int_0^1\dprod{\F(u^n +t(u^{n+1}-u^n)),u^{n+1}-u^n}\dt;
\end{align}
in particular, in view of the discussion above, we aim to maximize the integral on the right-hand side. For that purpose, we will employ a Taylor expansion of the integrand at $t=0$, provided that $\F:\,K \to Y^\star$ from~\eqref{eq:FH'} is Fr\'{e}chet differentiable. Specifically, let us first define the (continuously differentiable) function
\[
\psi_n(t):=\dprod{\F(u^n +t(u^{n+1}-u^n)),u^{n+1}-u^n},\qquad t\in[0,1].
\]
Then, if the difference $u^{n+1}-u^n$ is sufficiently small, applying a Taylor expansion at $t=0$ yields
\begin{align*}
\psi_n(t) \approx \psi_n(0)+\psi_n'(0)t=\dprod{\F(u^n),u^{n+1}-u^n}+t\dprod{\F'(u^n)(u^{n+1}-u^n),u^{n+1}-u^n}.
\end{align*} 
Since~\eqref{eq:fixedpointkacanovdamped} implies that
$
u^{n+1}-u^n=-\delta(u^n) \rho^n$, for each $n\ge 0$,
we obtain
\begin{align} \label{eq:psiapprox}
\psi_n(t) \approx -\delta(u^n)\dprod{\F(u^n),\rho^n}+t \delta(u^n)^2\dprod{\F'(u^n)\rho^n,\rho^n},
\end{align}
where we have exploited the fact that the Fr\'{e}chet derivative $\F'(u^n)$ is a linear operator. Hence, by recalling~\eqref{eq:hdfifftaylor} and integrating~\eqref{eq:psiapprox} from $t=0$ to $t=1$, we find that
\begin{align*}
\H(u^{n})-\H(u^{n+1}) \approx \delta(u^n)\dprod{\F(u^n),\rho^n}-\frac{\delta(u^n)^2}{2}\dprod{\F'(u^n)\rho^n,\rho^n}. 
\end{align*}
Then, a simple calculation reveals that the right-hand side is maximized for the damping parameter
\begin{align} \label{eq:dptaylor}
\delta^n=\delta(u^n):=\frac{\dprod{\F(u^n),\rho^n}}{\dprod{\F'(u^n)\rho^n,\rho^n}}.
\end{align}
In account of \eqref{eq:lowerbound} and the lower bound $\delta_{\min}=\nicefrac{\alpha}{4 L_\H}$, this leads to the step size Algorithm~\ref{alg:Taylor}. We note that the stopping criterion in line 6 will be satisfied once the damping parameter is small enough, cf.~Corollary~\ref{cor:convergence}, i.e., the procedure terminates after finitely many steps; indeed, the stopping criterion is certainly satisfied once we reach $\delta^n=\delta_{\min}$. Moreover, we underline that the derivative $\F'(u^n)$ must be available in the step size Algorithm~\ref{alg:Taylor}, cf.~\eqref{eq:dptaylor}.

\begin{algorithm}
\caption{Step size control via Taylor expansion}
\label{alg:Taylor}
\begin{flushleft} 
\textbf{Input:} Given $u^n \in K$, a correction factor $\sigma \in (\nicefrac{1}{2},1)$, and a parameter $\theta\in (0,\nicefrac{1}{2}]$. 
\end{flushleft}
\begin{algorithmic}[1]
\State Solve the linear problem $\A(u^n)\rho^n=\F(u^n)$ for $\rho^n\in X$, cf.~\eqref{eq:Ktrad2}.
\State Compute $\delta^n$ from~\eqref{eq:dptaylor} and set $\delta^n\gets\max \left\{\delta_{\min},\delta^n\right\}$.
\Repeat  	
\State Compute $u^{n+1}:=u^{n}-\delta^n\rho^n$, cf.~\eqref{eq:fixedpointkacanovdamped}.
\State Set $ \delta^n \gets \max\{\sigma \delta^n,\delta_{\min}\}$.
\Until {$\H(u^{n})-\H(u^{n+1}) \ge \theta\min\left\{\alpha,L_\H\right\}\norm{u^{n+1}-u^n}_Y^2$}\\ 
\Return $u^{n+1}$.
\end{algorithmic}
\end{algorithm}

\subsection{Step size control via a prediction-correction strategy} \label{sec:ssdp}

We will present a second adaptive damping parameter selection procedure that is partially based on ideas from~\cite[\S 3.1]{Deuflhard:04}. This strategy is more `ad hoc' than the Taylor expansion approach above, however, it does not require the differentiability of the operator $\F=\H'$, cf.~\eqref{eq:FH'}. The idea is fairly straightforward: For a given correction factor $\sigma \in (\nicefrac{1}{2},1)$ and damping factor $\delta >0$, we compare the energy decay for the damping parameters $\delta$ and $\delta'=\sigma^p \delta$, where $p \in \{-1,1\}$ depends on the previous step; we note that $p=-1$ yields an increased step size, whereas $p=1$ decreases the damping parameter. If applying $\delta'$ results in a larger energy decay, then we choose the damping parameter to be $\delta'$ in the present and subsequent steps, with $p$ unchanged; otherwise, if $\delta$ outperforms $\delta'$, then $\delta$ is retained, however, in the next step we replace $p$ by $-p$. This leads to Algorithm~\ref{alg:PreCor}.

\begin{algorithm}
\caption{Step size control via prediction-correction strategy}
\label{alg:PreCor}
\begin{flushleft} 
\multiline{\textbf{Input:} Given $u^n \in K$, a damping parameter $\delta \geq \delta_{\min}$, an exponent $p \in \{-1,1\}$, a correction factor $\sigma \in (\nicefrac{1}{2},1)$, and a parameter $\theta\in (0,\nicefrac{1}{2}]$. }
\end{flushleft}
\begin{algorithmic}[1]
\State Let $C:=\theta\min\left\{\alpha,L_H\right\}$.
\State Solve the linear problem $\A(u^n)\rho^n=\F(u^n)$ for $\rho^n\in X$, cf.~\eqref{eq:Ktrad2}.
\If {$p=1$ and $\delta< \sigma^{-1}\delta_{\min}$}
\State Set $p \gets -1$.
\EndIf
\State Set $\delta':=\sigma^p \delta$ and compute $\widetilde u^{n+1}:=u^{n}-\delta'\rho^n$, cf.~\eqref{eq:fixedpointkacanovdamped}.
\If {$\H(u^{n})-\H(\widetilde{u}^{n+1}) \geq C \norm{u^{n}-\widetilde{u}^{n+1}}_Y^2$}
\State Compute $u^{n+1}:=u^{n}-\delta\rho^n$, cf.~\eqref{eq:fixedpointkacanovdamped}.
\If {$\H(\widetilde{u}^{n+1}) \leq \H(u^{n+1})$ or $\H(u^{n})-\H(u^{n+1}) < C\norm{u^{n}-u^{n+1}}_Y^2$}
\State Set $\delta\gets\delta'$ and $u^{n+1}\gets\widetilde{u}^{n+1}$.
\Else
\State Set $p\gets-p$.
\EndIf
\Else
\State Set $p\gets1$, $\delta\gets\delta'$, and $u^{n+1}:=\widetilde{u}^{n+1}$.
\While {$\H(u^{n})-\H(u^{n+1}) < C\norm{u^{n+1}-u^n}_Y^2$}	
\State Set $\delta\gets\sigma \delta$ and compute $u^{n+1}:=u^{n}-\delta\rho^n$, cf.~\eqref{eq:fixedpointkacanovdamped}.
\EndWhile
\EndIf \\
\Return $\delta$, $u^{n+1}$, and $p$.
\end{algorithmic}
\end{algorithm}

\section{Application to quasilinear diffusion models} \label{sec:dampedkacanov}

In this section, we discuss the weak formulations of the boundary value problem~\eqref{eq:scl} as well as of the Ka\v{c}anov iteration scheme~\eqref{eq:kacanovs}. In addition, an equivalent variational setting will be established. Furthermore, a series of numerical experiments in the framework of finite element discretizations will be presented.

\subsection{Sobolev spaces}

Let $Y:=\Hone(\Omega)$ be the standard Sobolev space of $\Ltwo(\Omega)$-functions with weak derivatives in $\Ltwo(\Omega)$. We endow $Y$ with the inner product 
\[
(u,v)_Y:=
\int_\Omega \nabla u \cdot \nabla v \dx+\int_\Omega uv \dx, \qquad u,v \in Y,
\]
and with the induced $\Hone$-norm $\norm{u}_Y:=\sqrt{(u,u)_Y}$, $u \in Y$. 
Moreover, consider the closed linear subspace $X:= \{w \in Y: w|_{\Gamma_1} = 0\} \subset Y$, where $w|_{\Gamma_1}$ denotes the trace of $w \in Y$ on the (non-vanishing) Dirichlet boundary part $\Gamma_1 \subset \partial \Omega$. We equip $X$ with the $\Hone$-seminorm $\norm{u}_X:=\norm{\nabla u}_{\Ltwo(\Omega)}$, for $u \in X$; owing to the Poincar\'e-Friedrichs inequality, we note that the norm $\|\cdot\|_X$ is equivalent to the norm $\norm{\cdot}_Y$ on $X$, i.e., there exists a constant $c>0$ such that $c \norm{u}_Y \leq \norm{u}_X \leq  \norm{u}_Y$ for all $u \in X$.
Finally, we consider the closed and convex subset 
\begin{equation}\label{eq:KPDE}
K:=\{w \in Y: w|_{\Gamma_1}=g \ \text{on } \Gamma_1\}\subset Y,
\end{equation}
with $g$ the Dirichlet boundary data from~\eqref{eq:kacanovsdirichlet}. Evidently, $K$ has property (K). In particular, if $\Gamma_1=\partial \Omega$ and $g\equiv 0$ on $\partial \Omega$, then  we may consider $Y=X=K=\Hone_0(\Omega)$ with the norm $\|\cdot\|_X$.

\subsection{Weak formulations} \label{sec:kacanovscl}

For any given $u \in K$, we define a (symmetric) bilinear form $a(u;\cdot,\cdot)$ on $Y \times Y$ by
\begin{align} \label{eq:sclbilinear}
a(u;v,w):=\int_\Omega \muf{u}\nabla v \cdot \nabla w \dx,\qquad v,w \in Y.
\end{align}
Moreover, we introduce the linear functional
\begin{equation}\label{eq:b}
\dprod{b,v}:=\int_\Omega fv \dx-\int_{\Gamma_2} hv \dx,\qquad v\in X,
\end{equation}
where $\dprod{\cdot,\cdot}$ denotes the duality pairing in $X^\star \times X$, with $X^\star$ signifying the dual space of $X$. If the source function $f \in \Ltwo(\Omega)$ and the Neumann boundary data $h \in \Ltwo(\partial \Omega)$, then we notice that $b \in X^\star$; incidentally, more general assumptions on the data can be made, see, e.g., \cite[Rem.~25.32]{Zeidler:90}. 

In terms of the above forms, the weak formulation of \eqref{eq:scl} reads as follows:
\begin{equation}\label{eq:weak}
\text{Find }u \in K: \qquad a(u;u,v)=\dprod{b,v} \qquad \forall v \in X.
\end{equation}
Furthermore, for given $u^n \in K$, $n\ge 0$, the weak form of the Ka\v{c}anov scheme~\eqref{eq:kacanovs} is to find $u^{n+1}\in K$ such that
\[
a(u^n;u^{n+1},v)=\dprod{b,v} \qquad \forall v \in X.
\]


The ensuing result follows from standard arguments.

\begin{proposition} \label{prop:acoercivebounded}\mbox{}
If the diffusion coefficient $\mu$ satisfies {\rm ($\mu$3)}, then the form $a(\cdot;\cdot,\cdot)$ from \eqref{eq:sclbilinear} is bounded in the sense that
\[
|a(u;v,w)| \leq M_\mu \norm{v}_Y \norm{w}_Y \qquad\forall u\in K,~\forall v,w \in Y.
\]
Moreover, there exists a constant $\alpha >0$ such that, for any $u \in K$, we have the coercivity property
\begin{align} \label{eq:alpha1}
a(u;v,v) \geq \alpha \norm{v}^2_Y \qquad\forall v \in X,
\end{align}
and, especially,
\begin{align} \label{eq:alpha2}
a(u;v-w,v-w) \geq \alpha \norm{v-w}^2_Y \qquad\forall v,w \in K.
\end{align}
\end{proposition}

%

\subsection{Variational framework}\label{sec:var}

We introduce the (nonlinear) functional $\G:\,K \to  \mathbb{R}$ by
\begin{equation}\label{eq:G}
\G(u):=\int_\Omega \psi\big(\left|\nabla u\right|^2\big) \dx,\qquad
\text{with}\quad 
\psi(s):=\frac12\int_0^s \mu(t) \dt,\quad s\ge 0. 
\end{equation}
For~$u\in K$, the G\^ateaux derivative of $\G$ is given by
\begin{align} \label{eq:Gderivativescl}
\dprod{\G'(u),v}
&=\int_\Omega 2 \psi'\big(\left|\nabla u\right|^2\big)\nabla u \cdot \nabla v \dx
=\int_\Omega \muf{u} \nabla u \cdot \nabla v \dx,
\end{align}
for all $v\in X$, i.e., $\G'(u)=a(u;u,\cdot)$ in $X^\star$. 

Now, introduce the (energy) potential $\H:\, K \to \mathbb{R}$ by 
\begin{equation}\label{eq:HPDE}
\H(u):=\G(u)-\dprod{b,u},
\end{equation} 
with $\G$ and $b$ from~\eqref{eq:G} and~\eqref{eq:b}, respectively. If the diffusion coefficient $\mu$ satisfies the estimates
\begin{align} \label{eq:muassumption}
m_\mu(t-s) \leq \mu(t^2)t-\mu(s^2)s \leq M_\mu (t-s), \qquad t \geq s \geq 0,
\end{align}
then the Lipschitz condition~\eqref{eq:lipschitz} and the strong monotonicity property~\eqref{eq:strongmonotonicity} can be deduced with $\nu=m_\mu$ and $L_\H=3 M_\mu$, cf.~\cite[Prop.~25.26]{Zeidler:90}. 

\begin{remark}\label{rem:computeconst}
We comment on the assumption~\eqref{eq:muassumption}:
\begin{enumerate}[(a)]
\item
It is easily shown that {\rm ($\mu$1)--($\mu$4)} implies~\eqref{eq:muassumption}, however, we emphasize that the latter assumption does not require $\mu$ to be decreasing nor differentiable. Yet, if $\mu$ is differentiable, then \eqref{eq:muassumption} implies ($\mu3$) and ($\mu4$).
\item If the diffusion coefficient $\mu$ is continuously differentiable, then we can (easily) compute the bounds in~\eqref{eq:muassumption} by taking into account the mean value theorem. In particular, we may set
\begin{align*}
m_\mu=\inf_{t \geq 0} \xi'(t) \qquad \text{and} \qquad M_\mu=\sup_{t \geq 0} \xi'(t),
\end{align*}
where $\xi(t)=\mu(t^2)t$.
\item Recall from {\rm($\mu4$)} that the continuous function $\phi(t)$, cf.~\eqref{eq:orlicz}, is strictly convex and increasing for $t\ge 0$ by {\rm($\mu3$)} with $\phi(0)=0$; we note that these properties relate to the class of Orlicz functions. In this aspect, our work links to the more general context of Orlicz type nonlinearities which have been studied, for instance, in~\cite{Diening:07}.
\end{enumerate}
\end{remark}

The following result is a direct consequence of Corollary~\ref{cor:1}.

\begin{proposition}
Suppose that the diffusion coefficient $\mu$ satisfies~\eqref{eq:muassumption}. Then, the functional $\H$ from~\eqref{eq:HPDE} has a unique minimizer $u^\star\in K$, cf.~\eqref{eq:KPDE}, which satisfies the weak formulation 
\begin{equation}\label{eq:variationalproblem1}
\int_\Omega \muf{u} \nabla u \cdot \nabla v \dx
=\int_\Omega fv \dx-\int_{\Gamma_2} hv \dx
\qquad\forall v\in X;
\end{equation}
i.e., $u^\star$ is the unique (weak) solution of~\eqref{eq:weak}.
\end{proposition}

%

\subsection{Convergence of the modified Ka\v{c}anov method}


Recall that the properties (A1)--(A4) as well as (K) are satisfied if the diffusion coefficient obeys the bounds~\eqref{eq:muassumption} by our analysis in the previous sections~\S\ref{sec:kacanovscl} and \S\ref{sec:var}, see, in particular, Proposition~\ref{prop:acoercivebounded} and~\eqref{eq:Gderivativescl}. Hence, the assumptions for the convergence results, cf.~Theorem~\ref{thm:convergenceNew} and Corollary~\ref{cor:convergence}, are fulfilled in the context of the quasilinear elliptic PDE~\eqref{eq:scl}, \emph{without} assuming ($\mu$2).

\begin{theorem} \label{thm:sclconvergence}
Assume that the diffusion coefficient satisfies the bounds~\eqref{eq:muassumption}. Then, the damped Ka\v{c}anov method~\eqref{eq:Knew}, with $\delta:K \to [\delta_{\min},\delta_{\max}]$ and $0<\delta_{\min}\leq \delta_{\max} < \nicefrac{2 \alpha}{3 M_\mu}$, converges to the unique weak solution $u^\star \in K$ of~\eqref{eq:variationalproblem1}.
\end{theorem}

\begin{remark}
We emphasize that the assumptions on the damping function $\delta$ from Theorem~\ref{thm:sclconvergence} are sufficient for the key inequality~\eqref{eq:keyinequalityeps} to hold, cf.~Corollary~\ref{cor:convergence}, however, they are not necessary. Indeed, as both step size methods from \S\ref{sec:stepsize} guarantee this inequality, they yield the convergence in the setting of Theorem~\ref{thm:sclconvergence} without the restriction on~$\delta$.   
\end{remark}

\subsection{Numerical experiments} \label{sec:numexp}
We will now perform a number of numerical tests for the modified Ka\v{c}anov method based on the different step size methods from \S\ref{sec:stepsize} in the context of the quasilinear elliptic boundary value problem~\eqref{eq:scl}. 

In all experiments, we let $\Omega:=(-1,1)^2 \setminus ([0,1] \times [-1,0])$ be a standard L-shaped domain in $\mathbb{R}^2$. We focus on homogeneous Dirichlet boundary conditions, i.e.,~$\Gamma_1=\partial \Omega$ and $g \equiv 0$, and therefore we set $Y:=X=K=\Hone_0(\Omega)$. Moreover, we consider the norm $\norm{\cdot}_Y:=\norm{\nabla(\cdot)}_{\Ltwo(\Omega)}$, so that we obtain $\alpha=m_\mu$ in~\eqref{eq:alpha1} and \eqref{eq:alpha2}. The source function~$f$ in~\eqref{eq:sclmain}, respectively the linear functional $b \equiv b(u)$ in the abstract analysis in \S\ref{sec:abstract}, is chosen such that the analytical solution of \eqref{eq:sclmain}--\eqref{eq:scldirichlet}, with  $g \equiv 0$ on the Dirichlet boundary $\Gamma_1=\partial \Omega$, is given by the smooth function $u^\star(x,y)=\sin(\pi x) \sin(\pi y)$. For the numerical approximation, we will use a conforming $\mathbb{P}_1$-finite element framework with a uniform mesh consisting of approximately $3 \cdot 10^6$
triangles. Throughout we set the correction factor in the adaptive step size algorithms to be $\sigma=0.9$. We have implemented our algorithms in Matlab, and solved the linear equations by means of the backslash operator. Furthermore, the errors to be illustrated in the figures below are taken with respect to the underlying exact \emph{discrete} solution, which, in each case, was determined with the aid of 1000 steps of the Zarantonello iteration with a suitable damping parameter, cf.~\cite[Prop.~5.1]{HeidWihler:2020a}.

\subsubsection{Monotonically decreasing diffusion} \label{sec:mondec} We consider the nonlinear diffusion coefficient $\mu(t)=\mu_1(t)=(t+1)^{-1} + \nicefrac{1}{2}$, for $t \geq 0$, see Figure~\ref{fig:mu}. 
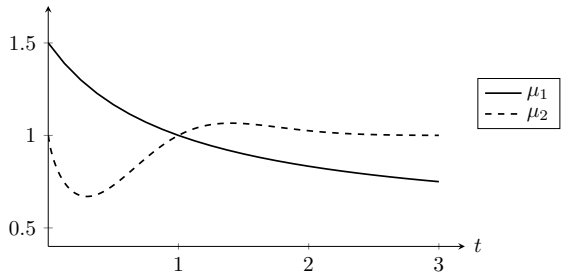
\begin{figure}
\begin{tikzpicture}[scale=0.8]
\begin{axis}[
		xscale = 1,
		yscale = 0.7,
        axis x line=middle, 
        axis y line=middle, 
        xmin=0, xmax=3.2, xlabel=$t$,
        xtick = {0,1,2,3},
        ymin=0.4, ymax=1.7, 
        ytick = {0.5,1,1.5},
        xlabel style={at={(axis cs:3.3,0.48)},anchor=north},
        legend pos=outer north east,
        ]
    \addplot[domain=0:3, thick] {1/(x+1)+0.5};
    \addplot[domain=0:3, thick, samples=200, dashed] {x*exp(-x^2)*ln(x+0.0001)+1};
    \legend{$\mu_1$,$\mu_2$};
\end{axis}
\end{tikzpicture}
\caption{The diffusion coefficients $\mu(t)=\mu_1(t)$ and $\mu(t)=\mu_2(t)$ from \S\ref{sec:mondec} and \S\ref{sec:nonmon}, respectively.}
\label{fig:mu}
\end{figure}
It is straightforward to verify that the diffusion parameter $\mu$ satisfies~\eqref{eq:muassumption} as well as the properties ($\mu$1)--($\mu$4). We compare the performance of the classical Ka\v{c}anov scheme~\eqref{eq:Ktrad} with the damped Ka\v{c}anov method~\eqref{eq:Knew} for both step size strategies from \S\ref{sec:stepsize}. For the application of the two step size methods, we recall that we need to know the values of the constants $m_\mu$ and $M_\mu$ a priori; in light of Remark~\ref{rem:computeconst} they are found to be $m_\mu = \nicefrac{3}{8} $ and $M_\mu= \nicefrac{3}{2}$. Moreover, here and in the two following experiments, we use the initial parameters $\delta=1$ and $p=-1$ in Algorithm~\ref{alg:PreCor}. Even though the diffusion parameter is monotonically decreasing and differentiable, which implies the convergence of the classical Ka\v{c}anov scheme, we can see from Figure~\ref{fig:mondec} that the damped Ka\v{c}anov method with either the step size method from \S\ref{sec:sstaylor} or \S\ref{sec:ssdp} performs (overall) better (in terms of error reduction per iteration step) than the undamped iteration. It is noteworthy that the damping parameters are larger than 1 in all steps for both approaches, see Figure~\ref{fig:mondecss}.

\begin{figure}
{\includegraphics[width=0.48\textwidth]{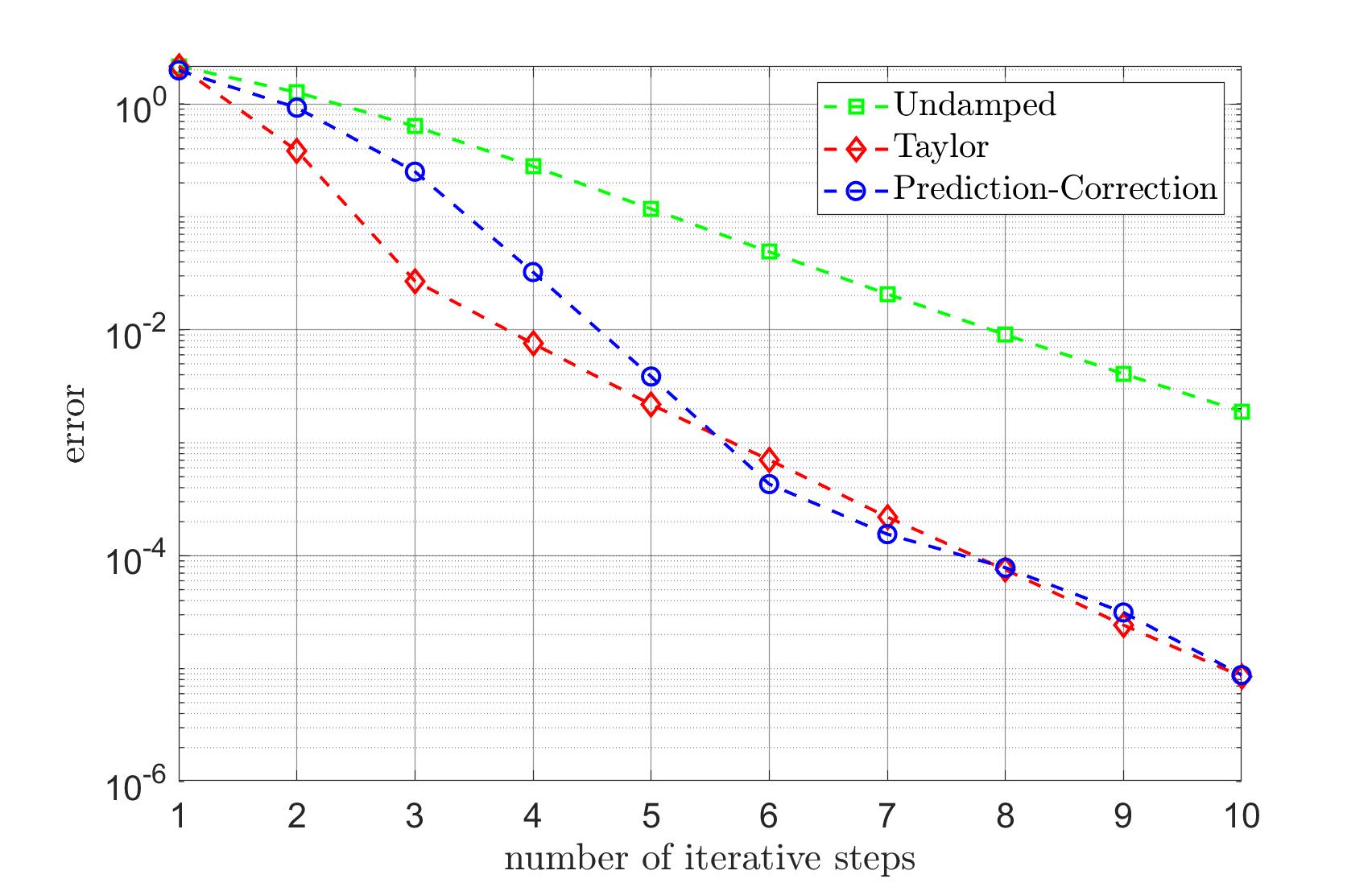}} \hfill
{\includegraphics[width=0.48\textwidth]{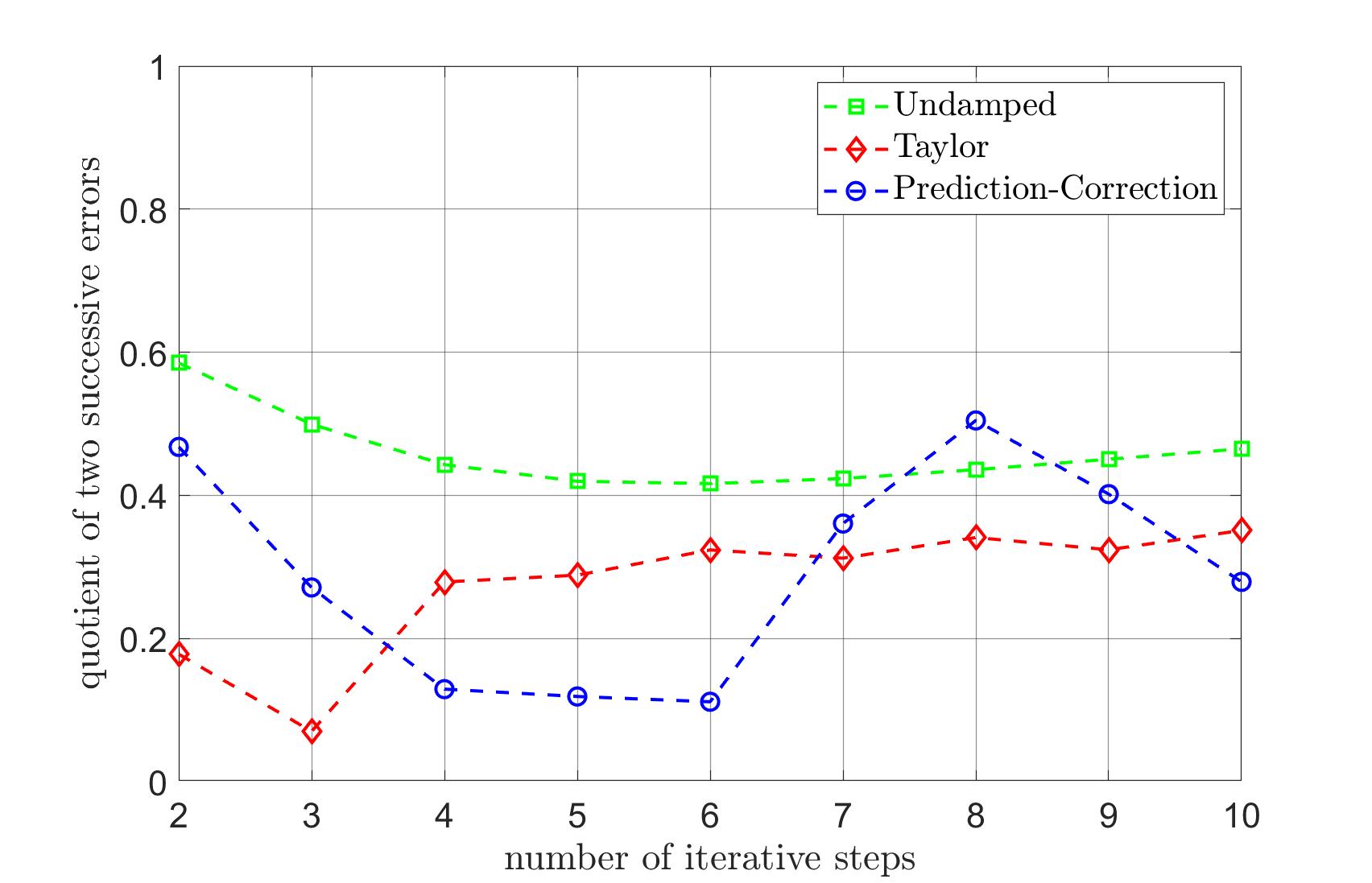}}
 \caption{Experiment~\ref{sec:mondec}. Comparison of the performance of the classical Ka\v{c}anov scheme ("Undamped") with the step size algorithms from \S\ref{sec:sstaylor} ("Taylor") and \S\ref{sec:ssdp} ("Prediction-Correction"). Left: Error decay. Right: Ratio of two successive errors.}\label{fig:mondec}
\end{figure}

\begin{figure}
{\includegraphics[width=0.5\textwidth]{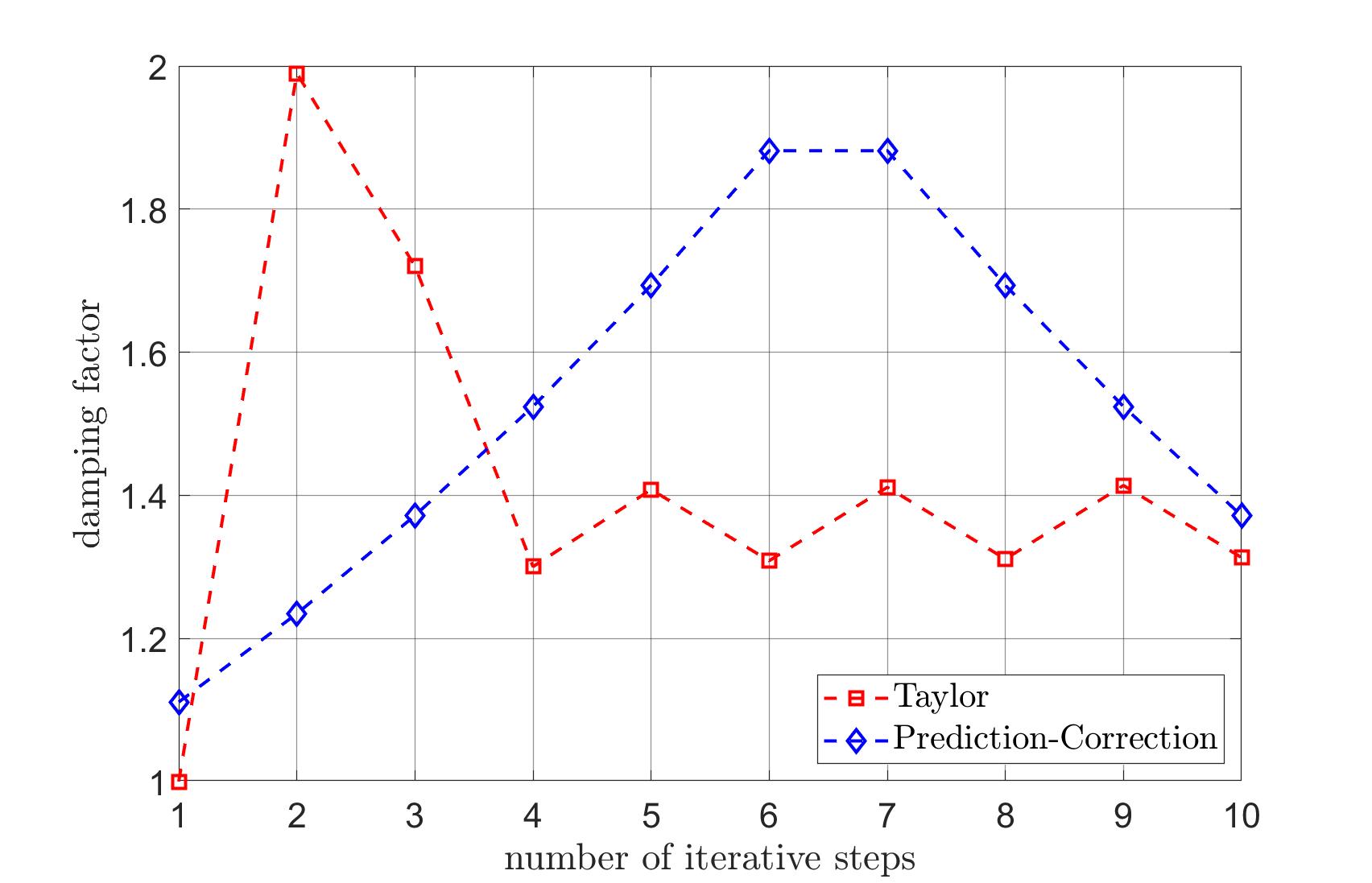}}
\caption{Experiment~\ref{sec:mondec}. Step sizes of each iterative step for the respective damped Ka\v{c}anov scheme.}
\label{fig:mondecss}
\end{figure}

\subsubsection{Non-monotone diffusion} \label{sec:nonmon}
In our second experiment, we consider the nonlinear diffusion parameter $\mu(t)=\mu_2(t)=t \exp(-t^2) \log (t+\epsilon)+1$, $t \geq 0$, for $\epsilon=10^{-4}$, see Figure~\ref{fig:mu}. It can be shown that~$\mu$ satisfies~\eqref{eq:muassumption} with $m_\mu \approx 0.483503$ and $M_\mu \approx 1.73565$, but is \emph{not} monotonically decreasing (nor increasing). Even though Figure~\ref{fig:nonmon} indicates that the classical Ka\v{c}anov scheme~\eqref{eq:Ktrad} may still converge, the convergence rate is really poor. In contrast, the modified Ka\v{c}anov scheme~\eqref{eq:Knew} with either damping strategy from \S\ref{sec:stepsize} exhibits a considerably better performance. We can observe in Figure~\ref{fig:nonmonss} that the damping parameters for both step size strategies from \S3 are (mostly) smaller than 1 in this specific experiment.

\begin{figure}
{\includegraphics[width=0.48\textwidth]{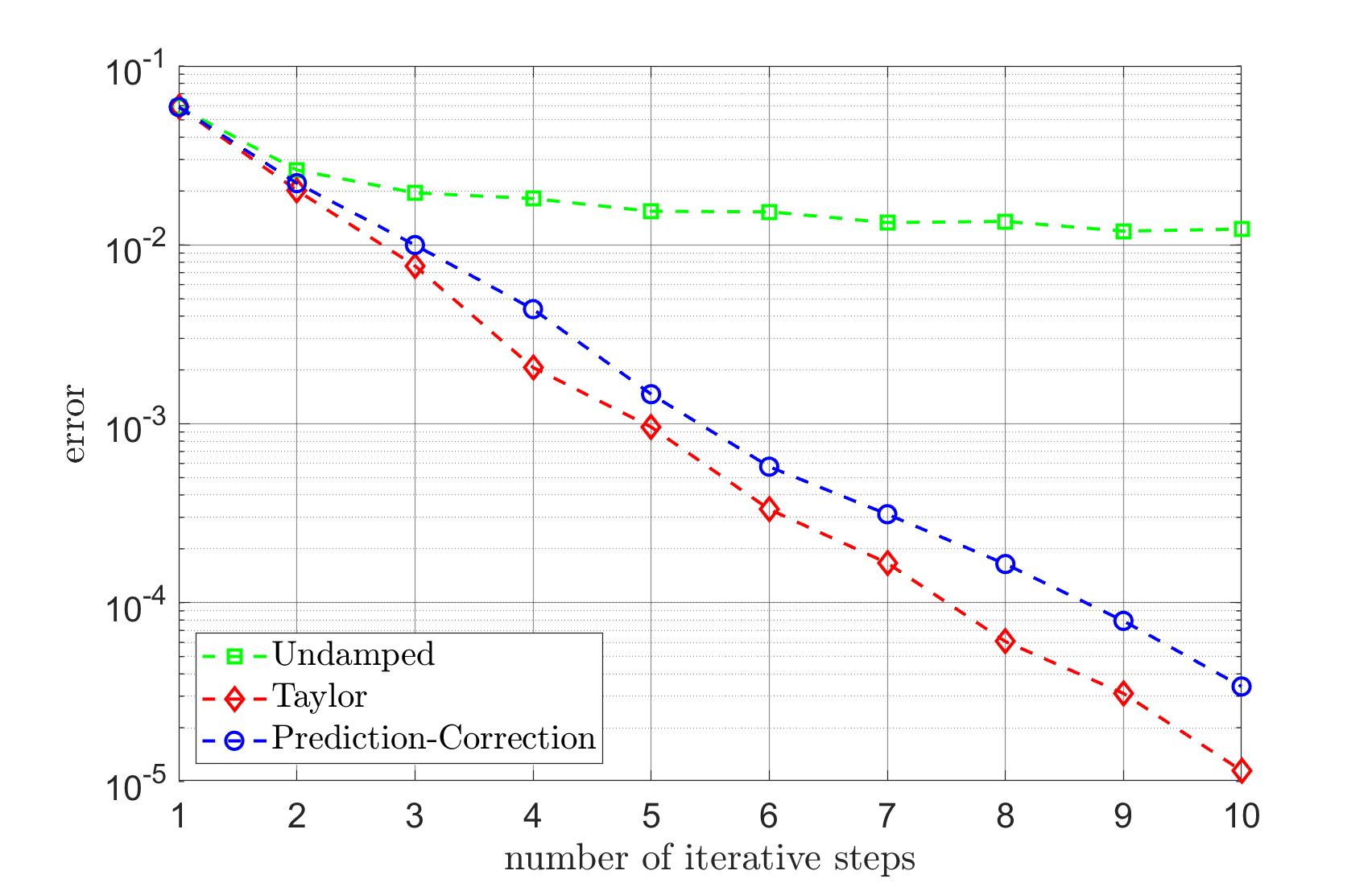}} \hfill
{\includegraphics[width=0.48\textwidth]{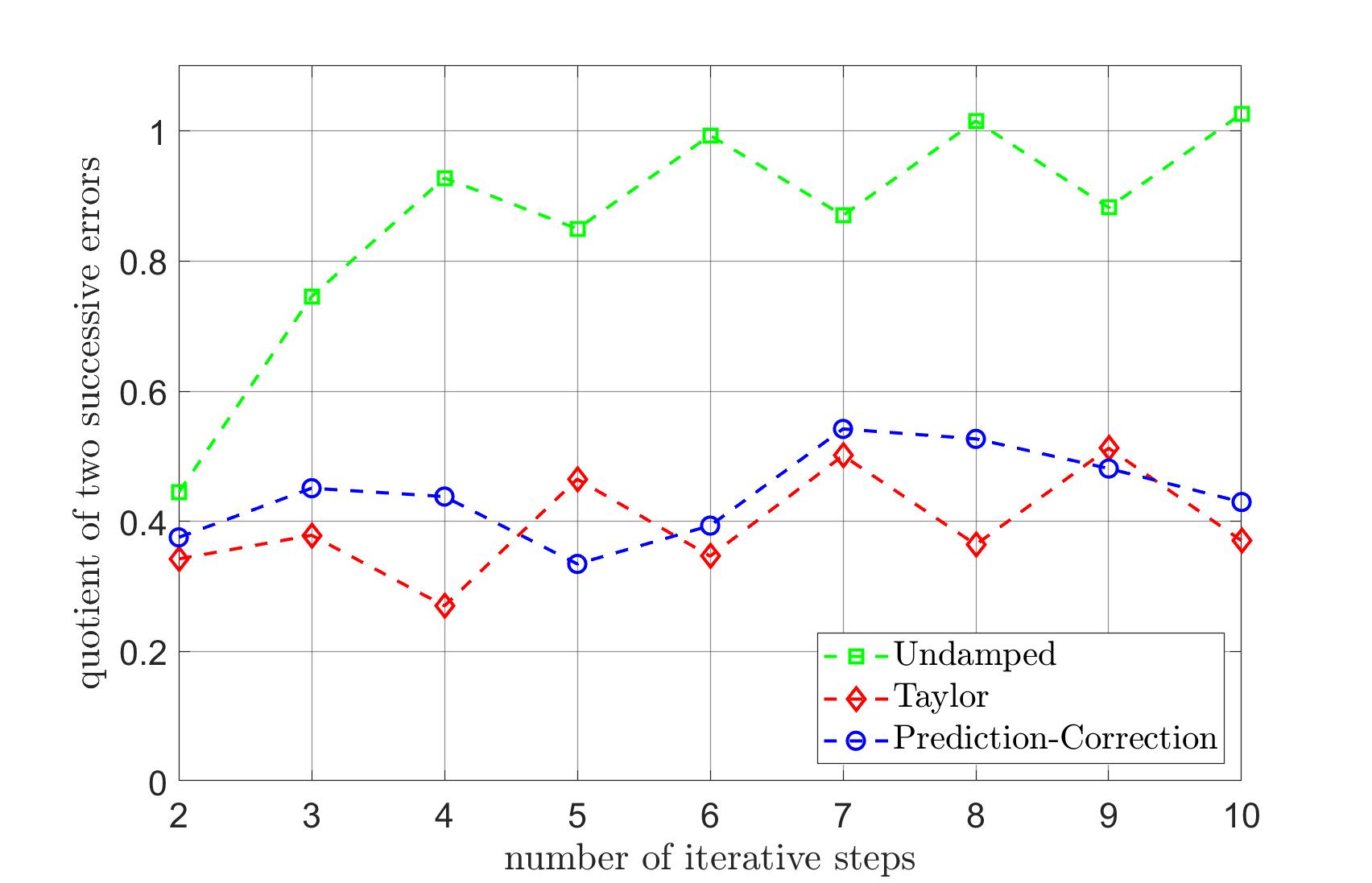}}
 \caption{Experiment~\ref{sec:nonmon}. Comparison of the performance of the classical Ka\v{c}anov scheme ("Undamped") with the step size algorithms from \S\ref{sec:sstaylor} ("Taylor") and \S\ref{sec:ssdp} ("Prediction-Correction"). Left: Error decay. Right: Ratio of two successive errors.}\label{fig:nonmon}
\end{figure}

\begin{figure}
{\includegraphics[width=0.5\textwidth]{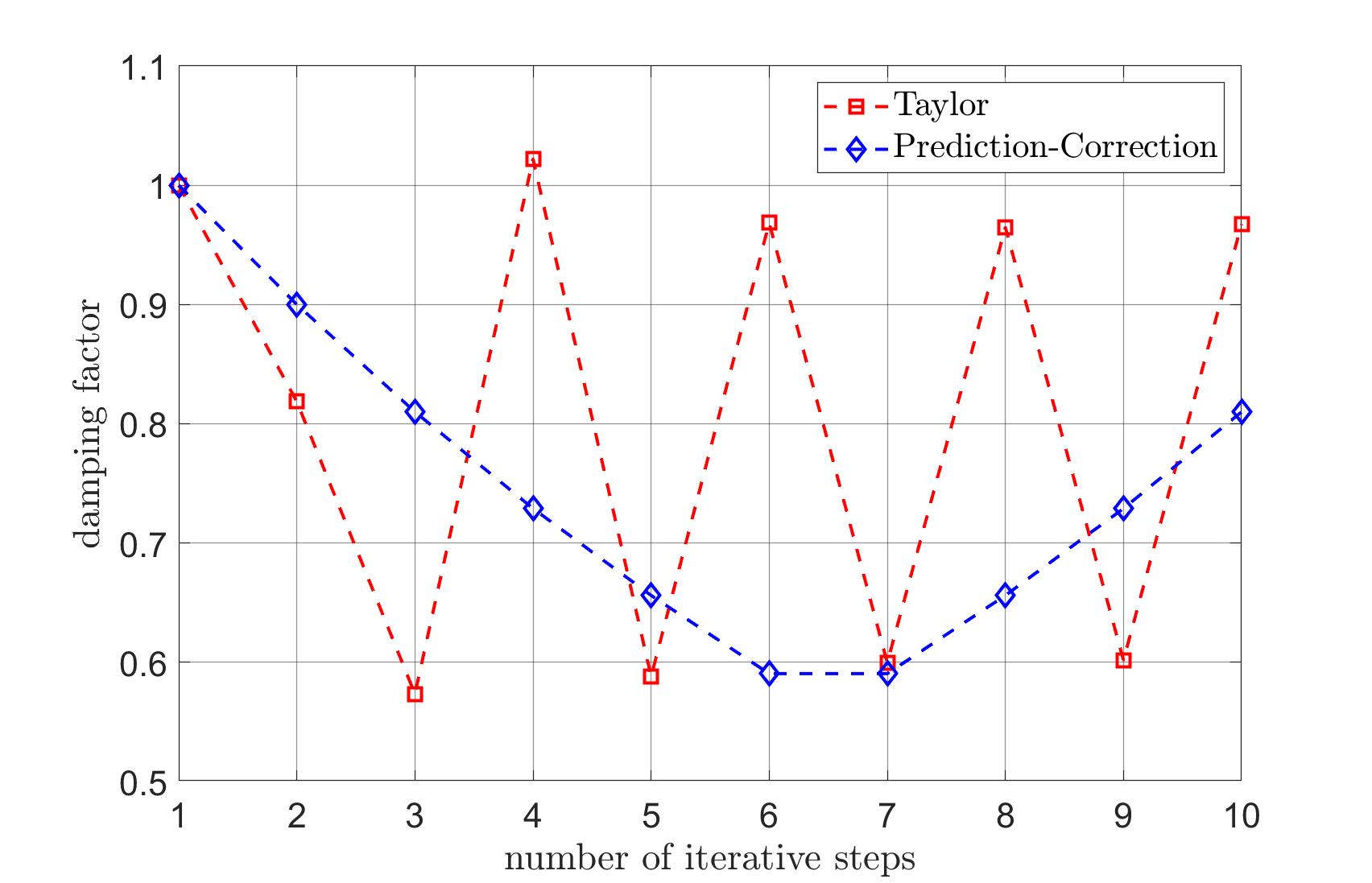}}
\caption{Experiment~\ref{sec:nonmon}. Step sizes of each iterative step for the respective damped Ka\v{c}anov scheme.}
\label{fig:nonmonss}
\end{figure}

\subsubsection{Diffusion coefficient motivated by fluid viscosities with shear thinning and shear thickening zones} \label{sec:viscosity}
In our last experiment, we will consider a diffusion coefficient that is motivated by a model of a fluid viscosity with both shear thinning and shear thickening zones; we refer to the work~\cite{shearthickening} for the details about the rheological properties of the corresponding fluids. Specifically, let
\begin{align*}
\mu_3(t)=\begin{cases}
\mu_c+\frac{\mu_0-\mu_c}{1+\left(\frac{t^2}{t-t_c}\right)^2}, & 0 \leq t \leq t_c \\
\mu_{\max}+\frac{\mu_c-\mu_{\max}}{1+\left(\frac{t-t_c}{t-t_{\max}}\right)^2}, & t_c<t\leq t_{\max} \\
\mu_\infty+\frac{\mu_{\max}-\mu_\infty}{1+\left(t-t_{\max}\right)^2}, & t>t_{\max},
\end{cases}
\end{align*}   
whereby we set $t_c=0.5, \ t_{\max}=2, \ \mu_0=5, \ \mu_c=4, \ \mu_{\max}=10$, and $\mu_{\infty}=6$; this diffusion coefficient is illustrated in Figure~\ref{fig:muViscosity}. 

\begin{figure}
\begin{tikzpicture}[scale=0.8]
\begin{axis}[
		xscale = 1,
		yscale = 0.7,
        axis x line=middle, 
        axis y line=middle, 
        xmin=0, xmax=4.5, xlabel=$t$,
        xtick = {0,0.5,1,2,3,4},
        ymin=3, ymax=12, 
        ytick = {4,5,6,8,10},
        xlabel style={at={(axis cs:4.7,3.48)},anchor=north},
        legend pos=outer north east,
        ]
    \addplot[domain=0:0.5, thick] {4+1/(1+((x^2)/(x-0.5)^2)};
    \addplot[domain=0.5:2, thick] {10-6/(1+((x-0.5)/(x-2))^2)};
    \addplot[domain=2:4.5, thick] {6+4/(1+(x-2)^2)};
\end{axis}
\end{tikzpicture}
\caption{The diffusion coefficients $\mu_3$ from \S\ref{sec:viscosity}.}
\label{fig:muViscosity}
\end{figure}
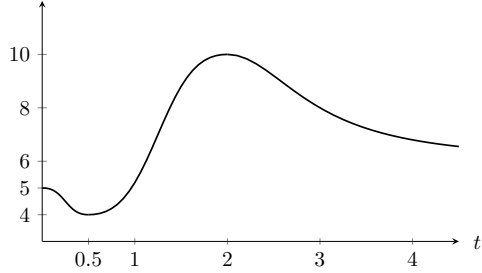

In~\cite{shearthickening} it is shown that the diffusion coefficient $\mu(t)$ is continuously differentiable. Once more, from Remark~\ref{rem:computeconst}, it follows that~\eqref{eq:muassumption} is satisfied with $m_\mu = 1.68$ and $M_\mu \approx 28.2696$. We clearly see in Figure~\ref{fig:viscosity} that the classical Ka\v{c}anov scheme \emph{does not converge} for this specific problem. In contrast, the modified Ka\v{c}anov method with either of the two step size strategies from \S\ref{sec:stepsize} converges perfectly, with the ratio of two successive errors being around $0.8$. Here, the step sizes in the damped Ka\v{c}anov scheme are, after an initial phase, between $0.3$ and $0.7$, and thus noticeably below $1$.    

\begin{figure}
{\includegraphics[width=0.48\textwidth]{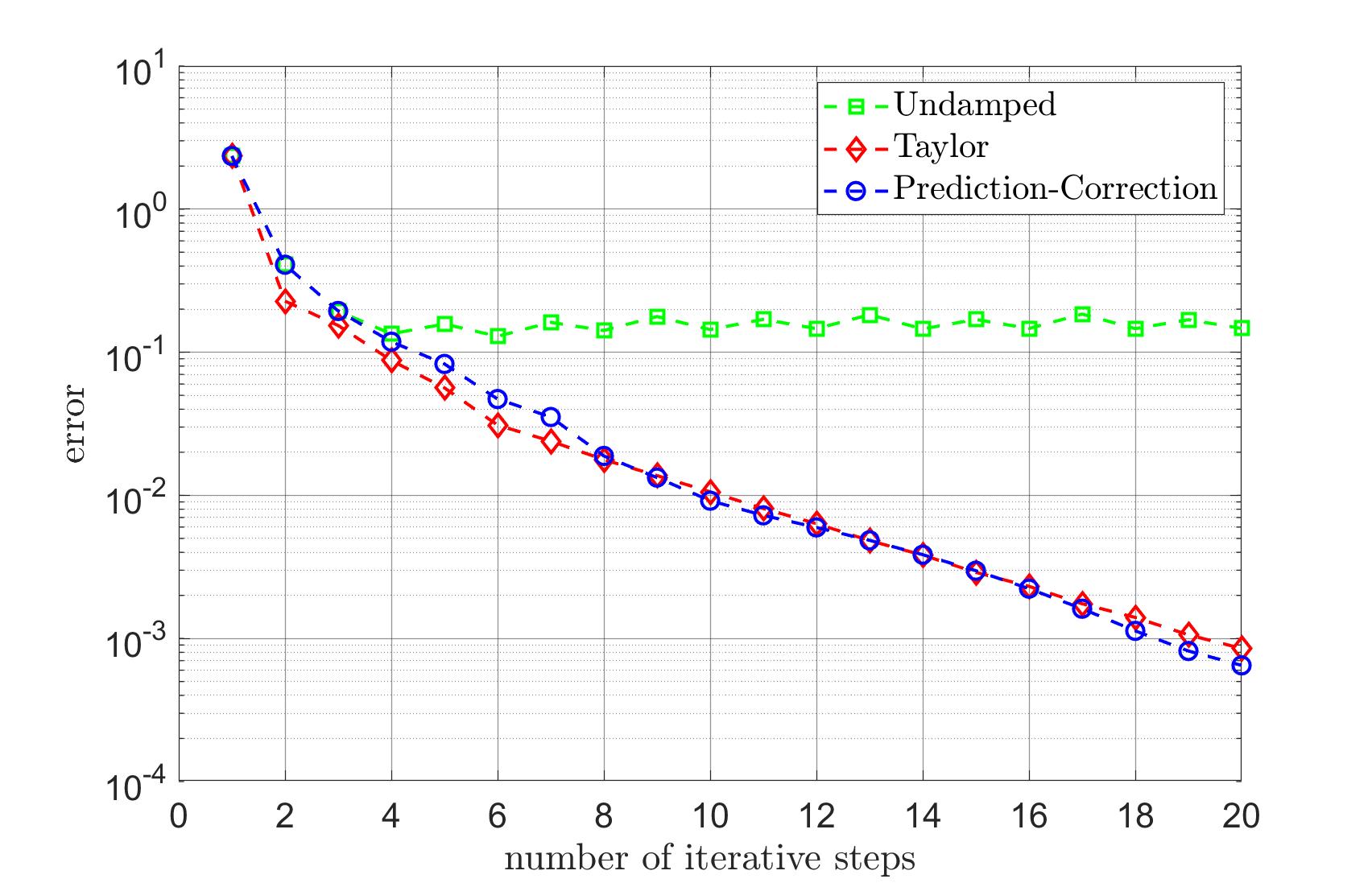}} \hfill
{\includegraphics[width=0.48\textwidth]{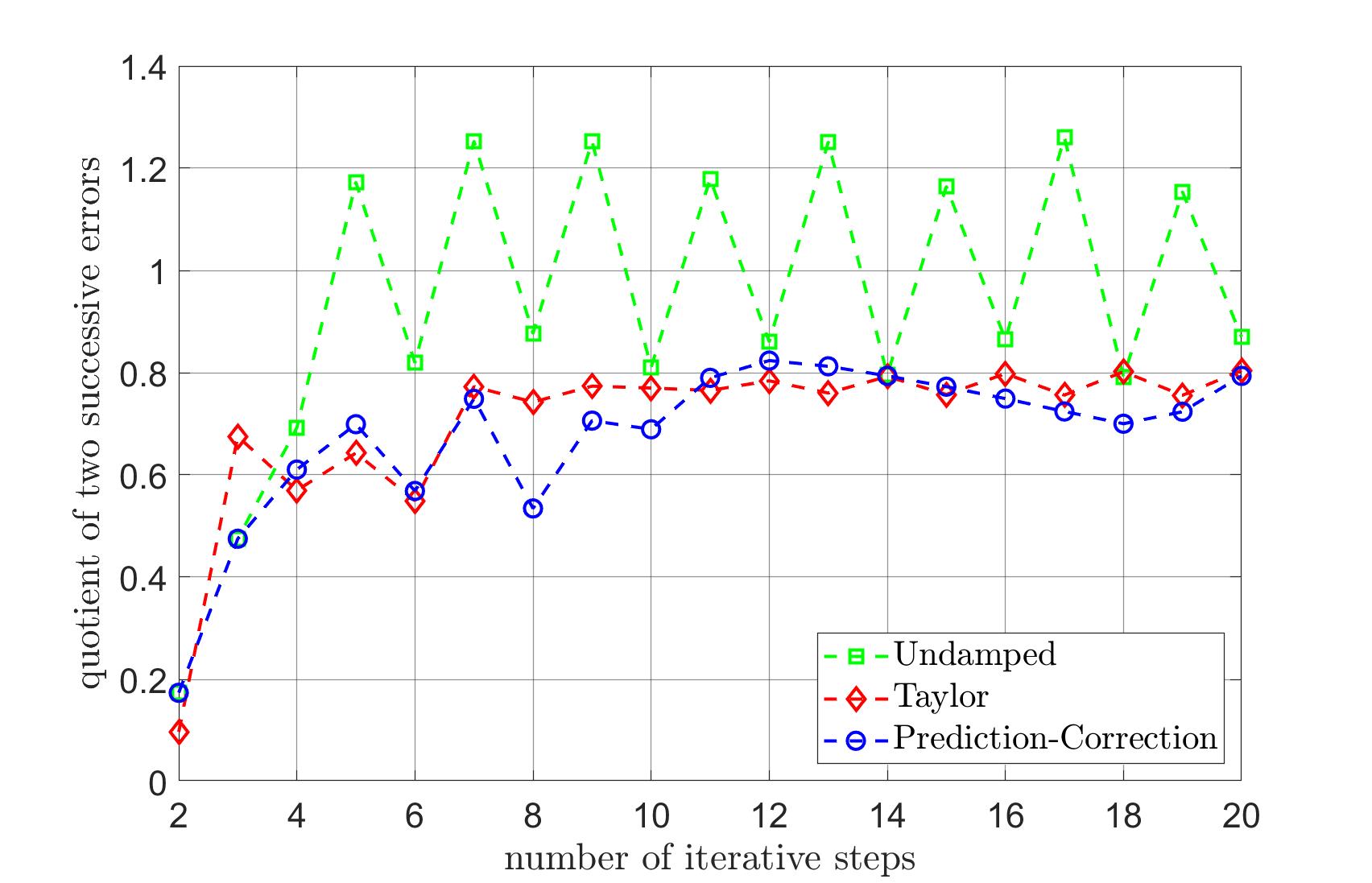}}
 \caption{Experiment~\ref{sec:viscosity}. Comparison of the performance of the classical Ka\v{c}anov scheme ("Undamped") with the step size algorithms from \S\ref{sec:sstaylor} ("Taylor") and \S\ref{sec:ssdp} ("Prediction-Correction"). Left: Error decay. Right: Ratio of two successive errors.}\label{fig:viscosity}
\end{figure}

\begin{figure}
{\includegraphics[width=0.5\textwidth]{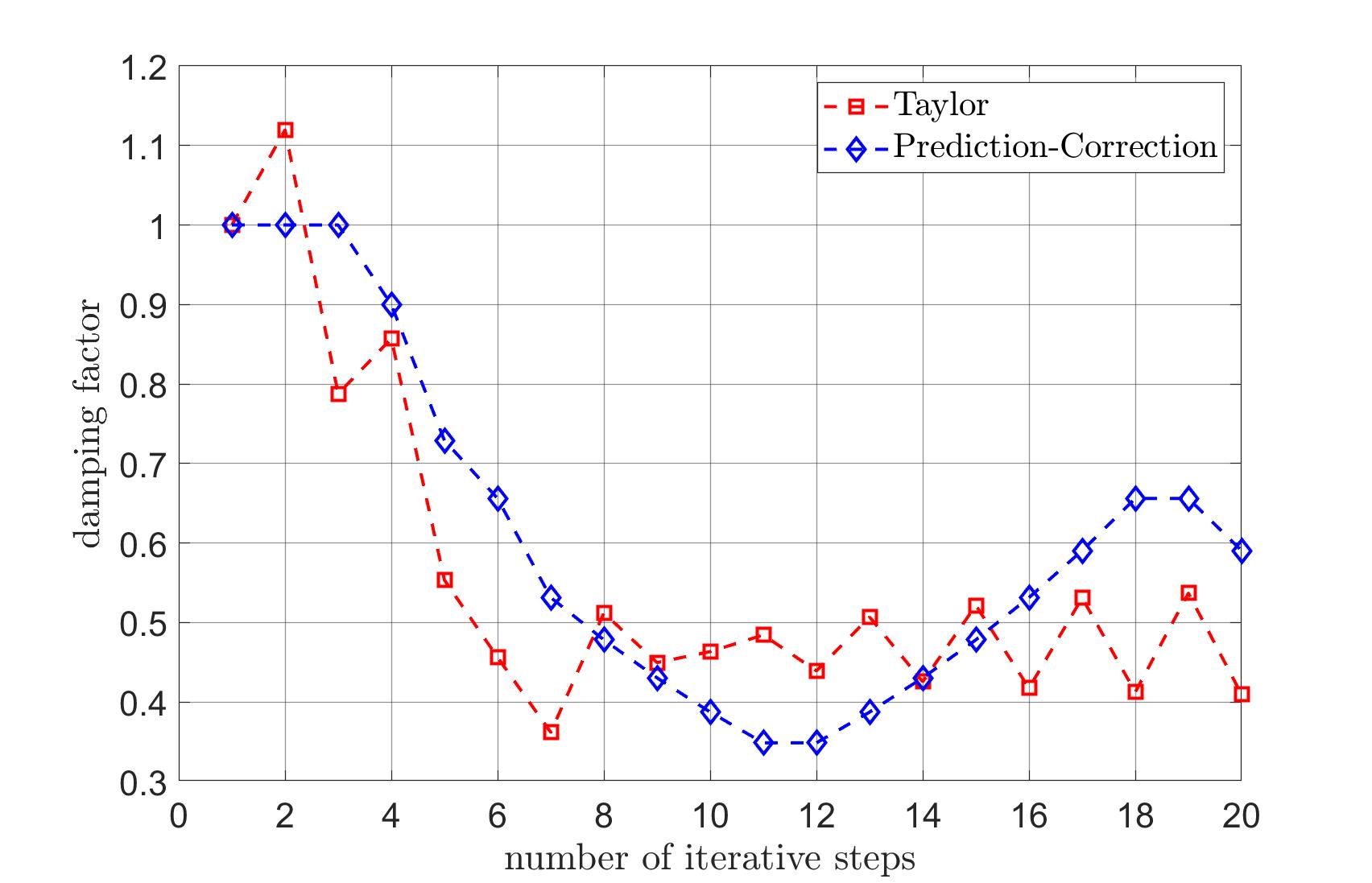}}
\caption{Experiment~\ref{sec:viscosity}. Step sizes of each iterative step for the respective damped Ka\v{c}anov scheme.}
\end{figure}

\section{Conclusion} \label{sec:conclusion}
In this work, we have devised a modified version of the classical Ka\v{c}anov iteration scheme. Exploiting the iterative linearization approach, cf.~\cite{HeidWihler:2020a}, we have shown that the introduction of a damping parameter allows to derive a new convergence analysis, which applies to a wider class of problems. For instance, in the context of quasilinear elliptic PDE, a standard monotonicity condition on the diffusion coefficient can be dropped. Moreover, our numerical tests highlight that the modified Ka\v{c}anov method, in combination with suitable damping strategies, outperforms the classical scheme for the examples under consideration. Especially, the final experiment in our work illustrates that the modified Ka\v{c}anov scheme can effectively approximate nonlinear problems, for which the classical Ka\v{c}anov method fails to generate a sequence converging to a solution. This underlines the relevance of our modified Ka\v{c}anov scheme. We close by remarking that our work can be extended in a straightforward manner to quasilinear systems with applications to, e.g., plasticity or quasi-Newtonian fluids. 

\bibliographystyle{amsalpha}
\bibliography{references}
\end{document}